%% file: article.tex
\documentclass[onefignum,onetabnum]{siamonline171218}

\input{shared}

\ifpdf
\hypersetup{
  pdftitle={Error bounds for dynamical spectral estimation},
  pdfauthor={R. J. Webber, E. H. Thiede, D. Dow, A. R. Dinner, J. Weare}
}
\fi

\myexternaldocument{supplement}

\begin{document}

\maketitle

\begin{abstract}
    Dynamical spectral estimation is a well-established numerical approach for estimating eigenvalues and eigenfunctions of the Markov transition operator from trajectory data. 
    Although the approach has been widely applied in biomolecular simulations,
    its error properties remain poorly understood.
    Here we analyze the error of a dynamical spectral estimation method called “the variational approach to conformational dynamics” (VAC). We bound the approximation error and estimation error for VAC estimates. Our analysis establishes VAC’s convergence properties and suggests new strategies for tuning VAC to improve accuracy.
\end{abstract}




\section{Introduction}

An essential goal in simulation studies is to identify functions
that decorrelate slowly in time. 
Since the values of these functions can be forecast far into the future,
they are used for dimensionality reduction and prediction.
Moreover, slowly decorrelating functions
describe many scientifically significant processes.
For example, in biomolecular systems,
large-scale arrangements that control
biological activity decorrelate slowly, compared to 
quickly fluctuating bond
lengths and angles.

Dynamical spectral estimation is a numerical method that identifies
slowly decorrelating functions
by estimating the
eigenfunctions and eigenvalues of the Markov transition operator of a system.
Under appropriate assumptions, a small number of eigenfunctions span the most slowly decorrelating functions of the system, and the associated eigenvalues determine the slowest decorrelation rates. 
Dynamical spectral estimation uses simulated trajectories to estimate these quantities of interest.

Despite the wide acceptance of dynamical spectral estimation,
estimated eigenfunctions and eigenvalues can have substantial error
\cite{swope2004describing},
and the cause of this error is not yet fully understood.
Our goal here is to identify and bound the major error sources, thereby identifying opportunities where dynamical spectral estimation can produce accurate results.

Dynamical spectral estimation has been used in fields as diverse as biomolecular simulation \cite{plattner2017complete}, fluid mechanics \cite{taira2017modal}, and geophysical analysis \cite{froyland2007detection}.
The approach goes by many names in the literature,
including Markov state models \cite{schutte1999direct}, time-lagged independent component analysis \cite{schwantes2013improvements},
Ulam's method \cite{ulam1960collection},
dynamical mode decomposition \cite{schmid2010dynamic},
and extended dynamical mode decomposition \cite{williams2015data}.
The methods are all closely related,
so an error analysis for any one of the methods
can shed useful light on the others.
Here, for concreteness, we focus on a dynamical spectral estimation method 
that is well-established in chemistry,
called ``the variational approach to conformational dynamics" (VAC) \cite{noe2013variational, chodera2014markov, noe2017collective, husic2018markov}.

VAC can be applied to any Markov
process $X_t$ that is ergodic and reversible with respect to a distribution $\mu$.
Starting from a data set of simulated trajectories, 
VAC is applied in two steps.
First, the data set is used to estimate expectations
$C_{ij}\left(\tau\right) = \E_{\mu}\left[\phi_i\left(X_0\right) \phi_j\left(X_{\tau}\right)\right]$
involving a set of basis functions $\left(\phi_i\right)_{1 \leq i \leq n}$.
Then, the spectral decomposition of the matrix $C\left(0\right)^{-1} C\left(\tau\right)$
is used to estimate eigenvalues and eigenfunctions of the transition operator of $X_t$.

Our mathematical analysis 
establishes bounds on VAC's approximation error
and estimation error.
\emph{Approximation error}
is the error in eigenvalue and eigenfunction estimates if the expectations
$C_{ij}\left(\tau\right) = \E_\mu\left[\phi_i\left(X_0\right) \phi_j\left(X_\tau\right)\right]$
are computed perfectly.
\emph{Estimation error}
is the additional error incurred
in VAC estimates because matrices $C\left(0\right)$ and $C\left(\tau\right)$ 
are computed imperfectly using a finite data
set.

We are not the first authors to mathematically examine VAC's error.
Djurdjevac and coauthors \cite{djurdjevac2012estimating} bounded
the approximation error of VAC eigenvalues.
We extend their work by 
bounding the approximation error for VAC eigenfunctions, which are the chief objects of interest
in most applications of dynamical spectral estimation.
Additionally, we provide the
first analysis of estimation error both for VAC eigenvalues and for eigenfunctions.

Our analysis of VAC also requires proving original error bounds.
Standard bounds
for the approximation of eigenspaces
(e.g., \cite[pg. 103]{saad2011numerical} or \cite[pg. 990]{knyazev1997new})
depend on the inverse gap
between eigenvalues.
However, the gap between any two non-trivial eigenvalues of the transition operator vanishes exponentially fast
with the lag time parameter $\tau$.
Therefore, the standard bounds increase exponentially as $\tau \rightarrow \infty$.
In contrast to this asymptotic scaling,
we contribute a sharp new perturbation bound that
depends only on the inverse \emph{relative} gap between eigenvalues.
This new bound reaches its minimal value
in the large $\tau$ limit,
demonstrating the benefit of long lag times
for reducing approximation error.
In contrast,
our asymptotic expressions for the estimation error do depend on the inverse spectral gap
and therefore grow in the large $\tau$ limit.
Therefore,
it is best to select an intermediate lag time.

While there is no single ideal lag time dictated by our analysis,
we offer new tools for tuning VAC to reduce the estimation error.
One such tool, the \emph{VAC condition number}, 
identifies the subspaces of VAC eigenfunctions most sensitive to estimation error.
A second diagnostic, the \emph{mean squared estimation error}, 
identifies the typical size of the estimation error at different lag times.
We provide data-driven formulas for calculating these quantities,
enabling VAC users to identify and avoid the most problematic subspaces and lag times.
Our experiments confirm that using these diagnostic tools leads to improved accuracy for VAC estimates.


The paper is organized as follows.
Background material is given in \cref{sec:background},
theoretical results are in
\cref{sec:main}, 
numerical experiments are in \cref{sec:experiments},
mathematical derivations are in \cref{sec:math}, and the conclusions follow in
\cref{sec:conclusions}.


\section{Background}
\label{sec:background}

This section presents background material
explaining the VAC algorithm and the dynamical quantities VAC approximates.


\subsection{VAC}{\label{sub:VAC}}

We begin by introducing the steps of VAC
when the algorithm is applied to trajectory data from a 
Markov process $X_t$ with an ergodic, reversible distribution $\mu$.
The algorithm starts by estimating expectations
involving a set of basis functions $\left(\phi_i\right)_{1 \leq i \leq n}$.
Subsequently, VAC solves an eigenvalue problem involving matrices of expectations.

\begin{algorithm}
\caption{VAC algorithm at lag time $\tau$.}
\algsetup{
linenosize=\small,
linenodelimiter=.
}
\label{alg:generalvac}
\begin{algorithmic}[1]
\STATE{Form matrix $\hat{C}\left(0\right)$
with entries $\hat{C}_{ij}\left(0\right) \approx C_{ij}\left(0\right) = \E_\mu\left[\phi_i\left(X_0\right) \phi_j\left(X_0\right)\right]$.}
\STATE{Form matrix $\hat{C}\left(\tau\right)$ with entries
$\hat{C}_{ij}\left(\tau\right) \approx C_{ij}\left(\tau\right) = \E_\mu\left[\phi_i\left(X_0\right) \phi_j\left(X_\tau\right)\right]$.}
\STATE{Solve eigenvalue problem
$\hat{\lambda}_i^{\tau} \hat{v}^i\left(\tau\right) = \hat{C}\left(0\right)^{-1} \hat{C}\left(\tau\right) \hat{v}^i\left(\tau\right)$.}
\STATE{Return VAC eigenvalues $\hat{\lambda}_i^{\tau}$ and VAC eigenfunctions $\hat{\gamma}_i^{\tau} = \sum_j \hat{v}_j^i \left(\tau\right) \phi_j$.}
\end{algorithmic}
\end{algorithm}

In \cref{alg:generalvac},
we are purposefully vague about the exact method for obtaining trajectory data to estimate
\begin{equation}
    \hat{C}_{ij}\left(\tau\right) \approx C_{ij}\left(\tau\right) = \E_\mu\left[\phi_i\left(X_0\right) \phi_j\left(X_\tau\right)\right].
\end{equation}
One common approach involves simulating long trajectories $X_t$
and removing the start of each trajectory
to limit equilibration bias \cite{sokal1997monte}.
A second common approach (``importance sampling" \cite{liu2008monte})
involves simulating short trajectories
and addressing bias through an appropriate reweighting procedure
\cite{nuske2017markov,wu2017variational}.
Since there are no restrictions on how the data set is generated, enhanced sampling techniques can
be used to generate the trajectory initial conditions or even the trajectories themselves \cite{buchete2008peptide, pan2008building}.

In addition to collecting a data set,
another key design feature affecting VAC is the choice of basis functions.
In the mid-1990s, early versions of VAC used the coordinate axes as basis functions \cite{takano1995relaxation, hirao1997molecular},
a choice that remains common in molecular dynamics simulations \cite{naritomi2011slow, schwantes2013improvements, perez2013identification}.
Independently, in the late 1990s and early 2000s, researchers began constructing spectral estimates using ``Markov state models" \cite{schutte1999direct, swope2004describing, swope2004b}, 
a procedure mathematically equivalent to performing VAC using
a basis of indicator functions on a partition of state space.
This idea of using a basis of indicator functions can be traced back to a publication by Stanislaw Ulam in 1960
\cite[pg. 74-75]{ulam1960collection}
and leads to simplifications in the eigenvalue problem in \cref{alg:generalvac}.
In the 2010s, it was observed that these schemes shared a common mathematical framework that could be extended to arbitrary basis sets \cite{noe2013variational}.
Subsequent work led to the development of new families of basis functions
\cite{nuske2014variational, vitalini2015basis, boninsegna2015investigating, nuske2016variational}.

The name ``variational approach to conformational dynamics"
is inspired by the min-max
principle for self-adjoint operators
\cite{noe2013variational, reed1978methods}.
This variational principle demonstrates that 
the top eigenfunctions $\eta_1, \ldots, \eta_k$ 
of the transition operator maximize the value of the autocorrelation function
\begin{equation}
    \rho_\eta\left(\tau\right) = \corr_{\mu}\left[\eta\left(X_0\right), \eta\left(X_\tau\right)\right]
\end{equation}
at any lag time $\tau \geq 0$.
Thus, when $\eta$ is a linear combination of the top $k$ eigenfunctions
and $u$ is uncorrelated with the top $k$ eigenfunctions, the autocorrelation functions are related by
\begin{equation}
    \rho_{\eta}\left(\tau\right) \geq \rho_u\left(\tau\right),
    \quad \tau \geq 0.
\end{equation}
Consistent with this variational principle, 
VAC constructs linear combinations of basis functions
that maximize autocorrelations.
A recent approach due to
\cite{mardt2017vampnets}
and \cite{chen2019nonlinear}
extends the linear fitting procedure in VAC
by using artificial neural networks
to maximize autocorrelations.
However, in the present analysis
we focus on the linear VAC algorithm
as described in \cref{alg:generalvac},
and we leave analysis of the nonlinear fitting procedure to future work.

To help clarify the relationship between VAC and other algorithms,
we observe that the computational steps in \cref{alg:generalvac} 
can be used for many purposes.
For example,
AMUSE
\cite{tong1991indeterminacy, molgedey1994separation}
uses the same computational procedure
as \cref{alg:generalvac},
but the goal is to solve the blind-source separation
problem in signal processing.
Likewise, dynamic mode decomposition \cite{rowley2009spectral}
and extended dynamic mode decomposition \cite{williams2015data}
use the same computational procedure as
\cref{alg:generalvac},
but the goal is to analyze non-reversible processes,
particularly deterministic fluid flows.
While the underlying computations are similar
in all these cases,
VAC refers specifically to the spectral estimation of time-reversible processes.
To learn more about the connections between VAC
and other related algorithms,
we refer the reader to the helpful review paper by Klus and coauthors \cite{klus2018data}.


\subsection{Spectral theory}
In this subsection, we take a closer look at the transition operator of the process $X_t$
and its eigenfunctions.
We assume $X_t$ is either a continuous-time Feller process \cite{kallenberg2006foundations}
or a discrete-time process restricted to even times $t = 0, 2, 4, \ldots$.
We assume $X_t$ is ergodic and time-reversible with respect to a distribution $\mu$.
We use $\left< \cdot, \cdot \right>$ to
denote the inner product on the Hilbert space $L^2\left(\mu\right)$,
and we set $\left\lVert \cdot \right\rVert = \left<\cdot, \cdot\right>^{1 \slash 2}$.
Lastly, we use $P_{\mathcal{U}}$ to denote the orthogonal projection \cite[pg. 187]{reed1975methods} onto the closed linear subspace $\mathcal{U}$
and
$P_f$ to denote the orthogonal projection onto the one-dimensional subspace spanned by the function $f$.

The transition operator \cite{kallenberg2006foundations}, also called the Koopman operator,
is defined
as the conditional expectation
operator satisfying
\begin{equation}
    T_t\left[f\right]\left(x\right) = \E\left[\left.f\left(X_t\right)\right|X_0 = x\right].
\end{equation}
There are three main properties of the transition operator
that determine information about its eigenfunctions.
\begin{enumerate}[leftmargin = *]
    \item{\label{item:property1}}
    The transition operator $T_t$ is self-adjoint in $L^2\left(\mu\right)$.
    The self-adjointness follows from the time-reversibility condition
    \begin{equation}\label{eq:reversibility}
        \mu\left(\mathop{dx}\right) p_t\left(x, \mathop{dy}\right)
        = \mu\left(\mathop{dy}\right) p_t\left(y, \mathop{dx}\right),
    \end{equation}
    where $p_t\left(x, \mathop{dy}\right)$ denotes the transition probabilities for the process $X_t$.
    By integrating over equation \eqref{eq:reversibility},
    we verify the self-adjointness property 
    \begin{equation} 
        \left<f, T_t g\right> = \left<T_t f, g\right>,
        \quad f,g \in L^2\left(\mu\right).
    \end{equation}
    
    \item{\label{item:property2}}
    The transition operator satisfies the semigroup property
    \begin{equation}
        T_{t+s} = T_t T_s.
    \end{equation}
    For discrete-time processes, the semigroup property guarantees a decomposition
    \begin{equation}
        T_t = \left(T_1\right)^t,
        \quad t = 0, 1, 2, \ldots
    \end{equation}
    For continuous-time Feller processes,
    the decomposition can be extended even further,
    leading to the formula
    \begin{equation} 
        T_t = e^{t A},
        \quad t \geq 0,
    \end{equation}
    which relates the semigroup $T_t$ to its infinitesimal generator $A$ \cite{kallenberg2006foundations}.    
    
    \item{\label{item:property3}}
    The transition operator $T_t$ is nonnegative, that is,
    \begin{equation} 
        \left<f, T_t f\right>
        = \left<T_{t \slash 2 } f, T_{t \slash 2} f\right> \geq 0,
        \quad f \in L^2\left(\mu\right),
    \end{equation}
    for all $t \geq 0$ if $X_t$ is a continuous-time process
    and for $t = 0, 2, 4, \ldots$ if $X_t$ is a discrete-time process.
\end{enumerate}

Using the spectral theorem 
for self-adjoint operators
\cite{reed1975methods},
we obtain a decomposition of either
$A$ or $T_2$.
Then, we extend this decomposition to the transition operator at all lag times $t \geq 0$ or
$t = 0, 2, 4, \ldots$.
The spectral decomposition takes the form
\begin{equation}{\label{eq:spectral}}
    T_t = \int_0^{\infty} e^{-\sigma t} \Pi\left(\mathop{d\sigma}\right),
\end{equation}
where $\Pi\left(\mathop{d\sigma}\right)$ is a projection-valued measure.

The spectral decomposition
completely determines the time correlations 
of the process $X_t$.
If the spectrum is
discrete,
then a finite set of orthonormal eigenfunctions are
responsible for all the slowest decorrelations of the process.
However,
if there is an essential spectrum  containing $\sigma = 0$,
then an infinite set of orthonormal functions decorrelate arbitrarily slowly \cite[pg. 236]{reed1975methods}.

To avoid the possibility of having an essential spectrum containing $\sigma = 0$,
it is sufficient to assume $T_t$ is compact.
Under compactness, the spectral decomposition takes the form
\begin{equation}
    T_t = \sum_{i=1}^\infty e^{-\sigma_i t} P_{\eta_i},
\end{equation}
where $e^{-\sigma_1 t} > e^{-\sigma_2 t} \geq e^{-\sigma_3 t} \geq \cdots$
are eigenvalues and
$\eta_1, \eta_2, \eta_3, \ldots$ 
are the associated eigenfunctions.
Since the process is ergodic,
$e^{-\sigma_1 t} = 1$ is a simple eigenvalue
of $T_t$
corresponding to the eigenfunction $\eta_1 = 1$.
\Cref{fig:slow_example}
shows additional examples of eigenfunctions for a compact transition operator $T_t$.

\begin{figure}[h!]
    \centering
    \includegraphics[scale = .35, clip]{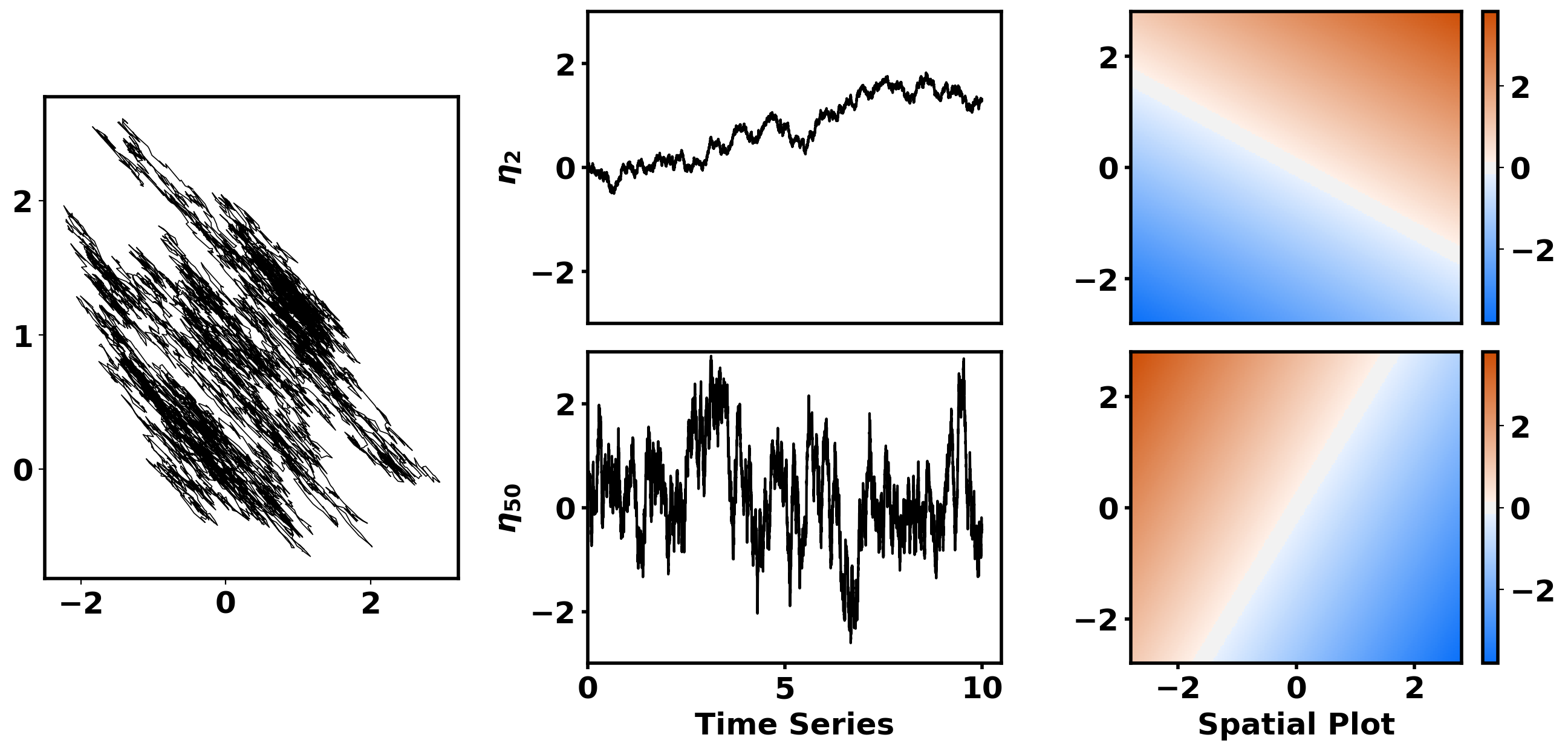}
    \captionsetup{singlelinecheck=off}
    \caption[.]{
    Eigenfunctions of a compact transition operator, corresponding to dynamics
    \begin{displaymath}
        \mathop{d}\begin{pmatrix} X_t \\ Y_t \end{pmatrix} 
    = \begin{pmatrix} -0.4 & 0.17 \\ 0.17 & -0.2 \end{pmatrix}
    \begin{pmatrix} X_t \\ Y_t \end{pmatrix}
    \mathop{dt} + \sqrt{2} \mathop{d} \begin{pmatrix} W^1_t \\ W^2_t \end{pmatrix}.
    \end{displaymath}
    Left: typical trajectory of $\left(X_t, Y_t\right)$.
    Upper middle: time series for eigenfunction $\eta_2$ with slow decorrelation timescale $\sigma_2^{-1} = 5$. 
    Lower middle: time series for eigenfunction $\eta_{50}$ with
    fast decorrelation timescale $\sigma_{50}^{-1} = 0.1$.
    Right: spatial structure of $\eta_2$ and $\eta_{50}$.}
    \label{fig:slow_example}
\end{figure}

While the compactness assumption
is enough to facilitate a rigorous analysis of VAC,
the compactness assumption can be overly restrictive.
In the Monte Carlo literature, 
there are numerous examples of transition operators that are not compact,
such as the transition operator for the Metropolis-Hastings sampler 
\cite{metropolis1953equation, atchade2007geometric}.
Therefore, we prefer to use the quasi-compactness assumption,
a weaker assumption satisfied by a broader class of processes.
\begin{assumption}[Quasi-compactness]{\label{ass:quasicompact}}
The spectral decomposition for the transition operator $T_t$ takes the form:
\begin{align}
\label{eq:modified_spectral}
    & T_t = \sum_{i=1}^r e^{-\sigma_i t} P_{\eta_i} + R_t,
    & R_t = \int_{\sigma_{r+1}}^{\infty} e^{-\sigma t} \Pi\left(\mathop{d\sigma}\right).
\end{align}
Here, $e^{-\sigma_1 t} > e^{-\sigma_2 t} \geq \cdots \geq e^{-\sigma_r t}$ are eigenvalues, $\eta_1, \eta_2, \ldots, \eta_r$ 
are the associated eigenfunctions, and $e^{-\sigma_{r+1} t}$ is not necessarily
an eigenvalue but it bounds the operator norm of the residual operator,
that is,
$\left\lVert R_t \right\rVert_2 \leq e^{-\sigma_{r+1} t}$.
\end{assumption}

\begin{remark}
In the analysis to follow, an ``eigenspace'' of $T_t$ denotes the closed
linear subspace of eigenfunctions
with a given eigenvalue.
An ``invariant subspace" $\mathcal{U}$
is any closed linear subspace
satisfying $T_t \mathcal{U} \subseteq \mathcal{U}$.
\end{remark}

\begin{remark}
There is a common modification of \cref{alg:generalvac}
where the estimated mean
$\hat{\mu}_i \approx \mu_i = \E_\mu\left[\phi_i\left(X_0\right)\right]$
is subtracted from each one of the
basis functions $\phi_i$
before performing VAC 
(see the discussion in \cite{klus2018data}).
When the mean is removed,
VAC no longer estimates the trivial eigenfunction $\eta_1 = 1$ with eigenvalue $e^{-\sigma_1 t} = 1$;
however, VAC continues to estimate all other eigenvalues and eigenspaces.
\end{remark}


\subsection{Approximation of eigenspaces}
It is colloquially said that VAC approximates eigenvalues
and eigenfunctions,
but it is more correct to say that VAC approximates eigenvalues
and \emph{eigenspaces}.
Recall that $\hat{\lambda}_i^{\tau}$ and $\hat{\gamma}_i^{\tau}$ are the VAC eigenvalues and eigenfunctions,
while $e^{-\sigma_i \tau}$ and $\eta_i$ are
the true eigenvalues and eigenfunctions of the transition operator.
We assume that VAC eigenvalues are arranged from largest to smallest so that
$\hat{\lambda}_1^{\tau} \geq \hat{\lambda}_2^{\tau} \geq \cdots \geq \hat{\lambda}_n^{\tau}$.
Then VAC approximates eigenvalues
\begin{equation}     \hat{\lambda}_i^\tau \approx e^{-\sigma_i \tau},
    \quad 1 \leq i \leq n.
\end{equation}
VAC approximates eigenspaces and other invariant subspaces
\begin{equation}
    \spa_{j \leq i \leq k} \hat{\gamma}_i^{\tau}
    \approx \spa_{j \leq i \leq k} \eta_i
\end{equation}
whenever there is a gap
between
$\left\{\sigma_j, \ldots, \sigma_k\right\}$ 
and all other $\sigma_i$ values.

To measure the error in VAC's invariant subspaces,
we introduce two distances: the gap distance 
$d_2\left(\cdot, \cdot\right)$
and the projection distance $d_{\F}\left(\cdot, \cdot\right)$ \cite{edelman1998geometry}.

\begin{definition}{\label{def:principal}}
Consider closed subspaces $\mathcal{U}$ and $\mathcal{W}$ and let $\mathcal{W}^{\perp}$ indicate the orthogonal complement of $\mathcal{W}$.
Then, we define the gap distance
and projection distance as follows:
\begin{align}
\label{eq:def}
    d_2\left(\mathcal{U}, \mathcal{W}\right) = \left\lVert P_{\mathcal{W}^{\perp}}
    P_{\mathcal{U}} \right\rVert_2,
    && d_{\F}\left(\mathcal{U}, \mathcal{W}\right) =
    \left\lVert P_{\mathcal{W}^{\perp}}
    P_{\mathcal{U}} \right\rVert_{\F}.
\end{align}
Here, $\left\lVert \cdot \right\rVert_2$
denotes the operator norm and
$\left\lVert \cdot \right\rVert_{\F}$
denotes the Hilbert-Schmidt norm,
also known as the Frobenius norm.
\end{definition}

The gap distance and projection distance are very flexible,
and definitions \eqref{eq:def}
can be applied even if
$\dim\left(\mathcal{U}\right) < \dim \left(\mathcal{W}\right) \leq \infty$.
In this case,
we observe that
$d_2\left(\mathcal{U}, \mathcal{W}\right)$
and 
$d_{\F}\left(\mathcal{U}, \mathcal{W}\right)$
are not technically distances.
Rather,
$d_2\left(\mathcal{U}, \mathcal{W}\right)$
and 
$d_{\F}\left(\mathcal{U}, \mathcal{W}\right)$
are properly interpreted as distances between
$\mathcal{U}$ and the nearest
$\dim\left(\mathcal{U}\right)$-dimensional subspace of $\mathcal{W}$.

We end this section by introducing a useful property of the projection distance,
which we apply repeatedly in the analysis.


\begin{lemma}{\label{lem:quadratic}}
Consider 
$\mathcal{U} = \spa\left(\mathcal{U}_1, \mathcal{U}_2\right)$
where $\mathcal{U}_1$ and $\mathcal{U}_2$ are orthogonal subspaces,
and $\mathcal{W} = \spa\left(\mathcal{W}_1, \mathcal{W}_2\right)$
where $\mathcal{W}_1$ and $\mathcal{W}_2$ are orthogonal subspaces.
Then,
\begin{equation}
    d_{\F}^2\left(\mathcal{U}_2, \mathcal{W}_2\right)
    \leq 
    d_{\F}^2\left(\mathcal{U}, \mathcal{W}\right) + d_{\F}^2\left(\mathcal{U}_1, \mathcal{W}_1\right).
\end{equation}
\end{lemma}
\begin{proof}
Calculate
\begin{align}
    d_{\F}^2\left(\mathcal{U}_2, \mathcal{W}_2\right)
    &= \left\lVert P_{\mathcal{U}_2}
    P_{\mathcal{W}^\perp}
    \right\rVert_{\F}^2
    + \left\lVert P_{\mathcal{U}_2}
    P_{\mathcal{W}_1}
    \right\rVert_{\F}^2 \\
    &\leq \left\lVert P_{\mathcal{U}}
    P_{\mathcal{W}^\perp}
    \right\rVert_{\F}^2
    + \left\lVert P_{\mathcal{U}_1^\perp}
    P_{\mathcal{W}_1}
    \right\rVert_{\F}^2 \\
    &= d_{\F}^2\left(\mathcal{U}, \mathcal{W}\right) + d_{\F}^2\left(\mathcal{U}_1, \mathcal{W}_1\right).
\end{align}
\end{proof}


\section{Theoretical results}
{\label{sec:main}}

To describe the approach taken in the theoretical analysis,
we introduce
an idealized VAC algorithm where
expectations
$C_{ij}\left(\tau\right) = \E_\mu\left[\phi_i\left(X_0\right) \phi_j\left(X_\tau\right)\right]$
and $C_{ij}\left(0\right) = \E_\mu\left[\phi_i\left(X_0\right) \phi_j\left(X_0\right)\right]$
are computed perfectly.
Notationally, we distinguish between VAC and idealized VAC
by using the $\hat{~}$ symbol
to indicate the quantities calculated using data.
For VAC,
we write 
$\hat{C}_{ij}\left(\tau\right)$, $\hat{v}^i\left(\tau\right)$, $\hat{\lambda}_i^{\tau}$,
and $\hat{\gamma}_i^{\tau}$.
For idealized VAC,
we write $C_{ij}\left(\tau\right)$, $v^i\left(\tau\right)$,
$\lambda_i^{\tau}$, 
and $\gamma_i^{\tau}$.



In the theoretical analysis, we use idealized VAC 
to isolate two different sources of error.
We decompose subspace error using
\begin{equation}
    \underbrace{d_{\F}\left(\spa_{j \leq i \leq k} \hat{\gamma}_i^{\tau},
    \spa_{j \leq i \leq k} \eta_i\right)}_{\text{total error}}
    \leq \underbrace{d_{\F}\left(\spa_{j \leq i \leq k} \gamma_i^{\tau},
    \spa_{j \leq i \leq k} \eta_i\right)}_{\text{approximation error}}
    + \underbrace{d_{\F}\left(\spa_{j \leq i \leq k} \hat{\gamma}_i^{\tau},
    \spa_{j \leq i \leq k} \gamma_i\right)}_{\text{estimation error}}.
\end{equation}
Analogously, we decompose eigenvalue error using
\begin{equation}
    \underbrace{\left|\hat{\lambda}_i^{\tau} - e^{-\sigma_i \tau}\right|}_{\text{total error}}
    \leq \underbrace{\left|\lambda_i^{\tau} - e^{-\sigma_i \tau}\right|}_{\text{approximation error}}
    + \underbrace{\left|\hat{\lambda}_i^{\tau} - \lambda_i^{\tau}\right|}_{\text{estimation error}}.
\end{equation}
Approximation error is the difference
between idealized VAC estimates
and the true
eigenvalues and eigenspaces.
Estimation error is the difference
between VAC estimates
and idealized VAC estimates.
We first present approximation error bounds in \cref{sub:approximation}
and then we present estimation error bounds in \cref{sub:estimation}.

\begin{remark}
To illustrate the implications of our error bounds, we use
numerical experiments.
Thus,
\Cref{fig:stabilizes} and \cref{fig:error_scaling}
demonstrate the error of VAC
when applied to the Ornstein-Uhlenbeck process
$\mathop{dX} = -X \mathop{dt} + \sqrt{2} \mathop{dW}$
using a basis of indicator functions.
Details on how the figures were generated
appear in the supplement.
\end{remark}


\subsection{Approximation error}{\label{sub:approximation}}
In this subsection,
we first bound the approximation
error by using traditional Rayleigh-Ritz approximation bounds.
However, we find that the Rayleigh-Ritz bounds 
do not provide enough information
to show how approximation error depends on the lag time parameter $\tau$.
Therefore, we derive improved bounds by
using original methods.
The improved bounds are asymptotically sharp at long lag times,
revealing that long lag times cause the approximation error to stabilize.


\subsubsection{Existing approximation bounds are inadequate}

The idealized VAC algorithm is equivalent to the \emph{Rayleigh-Ritz method} 
in spectral estimation.
In the Rayleigh-Ritz method \cite{stewart2001matrix}, 
the eigenvalues and eigenspaces of a target operator $A$ are estimated by introducing a subspace of functions $\mathcal{U}$ and then calculating the eigenvalues and eigenspaces of $P_{\mathcal{U}} \left.A\right|_{\mathcal{U}}$
where $\left.A\right|_{\mathcal{U}}$
denotes the restriction of $A$ to the subspace $\mathcal{U}$.
This is also exactly what is done in idealized VAC.
The target operator is the transition operator $T_{\tau}$,
and the subspace of basis functions is $\Phi = \spa_{1 \leq i \leq n} \phi_i$.
Moreover, the idealized VAC eigenfunctions $\gamma_i^{\tau}$
are eigenfunctions of $P_{\Phi} \left.T_{\tau}\right|_{\Phi}$
with eigenvalues $\lambda_i^{\tau}$.

The equivalence between the Rayleigh-Ritz
method and idealized VAC is known in the VAC literature
\cite{sarich2010approximation, djurdjevac2012estimating}.
However, the implications for VAC's approximation error have not yet been fully explored.
Djurdjevac and coauthors \cite{djurdjevac2012estimating} applied Rayleigh-Ritz error bounds to analyze
idealized VAC eigenvalues.
The following theorem takes a step further,
by also applying 
Rayleigh-Ritz error bounds to analyze idealized VAC eigenspaces.

\begin{theorem}[Approximation bounds]{\label{thm:knyazev}}
Fix the lag time $\tau > 0$ and the index $1 \leq k \leq r$,
but allow the basis set $\Phi$ to vary.
In the limit 
as $d_{\F}\left(\spa_{1 \leq i \leq k} \eta_i, \Phi\right) \rightarrow 0$,
the idealized VAC estimates converge as follows:
\begin{enumerate}[leftmargin = *]
\item 
The idealized VAC eigenvalues $1, 2, \ldots, k$ all converge
\begin{equation}
\label{eq:eigenvalue_convergence}
    \lambda_i^{\tau} \rightarrow e^{-\sigma_i \tau},
    \quad 1 \leq i \leq k.
\end{equation}
Moreover, the $k$th idealized VAC eigenvalue is bounded by
\begin{equation}
\label{eq:eigenvalue_bound}
    1 - d_2^2\left(\spa_{1 \leq i \leq k} \eta_i, \Phi\right)
    \leq \frac{\lambda_k^{\tau}}{e^{-\sigma_k \tau}} \leq 1.
\end{equation}
\item
When there is a gap between $\left\{\sigma_j, \ldots, \sigma_k\right\}$ and all other $\sigma_i$ values,
the subspace $\spa_{j \leq i \leq k} \gamma_i^{\tau}$ of idealized VAC eigenfunctions converges
\begin{equation}
    \spa_{j \leq i \leq k} \gamma_i^{\tau}
    \rightarrow 
    \spa_{j \leq i \leq k} \eta_i.
\end{equation}
Moreover, the top $k$ idealized VAC eigenfunctions are bounded by
\begin{equation}
\label{eq:subspace_bound}
	1 \leq \frac{d_{\F}^2\left(\spa\limits_{1 \leq i \leq k} \gamma_i^{\tau},
    \spa\limits_{1 \leq i \leq k} \eta_i \right)}
    {d_{\F}^2\left(\spa\limits_{1 \leq i \leq k} \eta_i, \Phi\right)}
    \leq 1 + \frac{\left\lVert P_{\Phi^{\perp}}
    T_{\tau} P_{\Phi} \right\rVert_2^2}
    {\left|e^{-\sigma_k \tau} - \lambda_{k+1}^{\tau}\right|^2}.
\end{equation}
\end{enumerate}
\end{theorem}
\begin{proof}
See \cite{knyazev1985sharp, knyazev1997new} for the original proofs,
or see the derivations in the supplement.
\end{proof}

The main takeaway from
\cref{thm:knyazev}
is that the approximation error converges to zero
in the limit as
\begin{equation}
\label{eq:projection_condition}
    d_{\F}\left(\spa_{1 \leq i \leq k} \eta_i, \Phi\right) \rightarrow 0.
\end{equation}
Condition \eqref{eq:projection_condition} implies that the basis set $\Phi$ must
become very rich, 
so that eigenfunctions $\eta_i$ can be closely approximated
using linear combinations of basis functions.

The Rayleigh-Ritz error bound \eqref{eq:subspace_bound} 
clearly indicates that the eigenspace approximation
error must decay with an increasingly rich basis.
However, the bound is not sufficiently detailed
to show how approximation error
depends on the lag time $\tau$.
As seen in \cref{fig:stabilizes},
the Rayleigh-Ritz bound \eqref{eq:subspace_bound}
asymptotes to infinity as the lag time increases,
implying that approximation error can grow arbitrarily large.
In contrast to this upper bound, however, experiments
reveal that approximation error decreases and then stabilizes as the lag time tends to infinity.
In the next section, we will derive an improved bound that is asymptotically sharp,
describing the exact behavior of the approximation error as $\tau \rightarrow \infty$.

\begin{figure}[h!]
    \centering
    \includegraphics[scale = .4, clip]{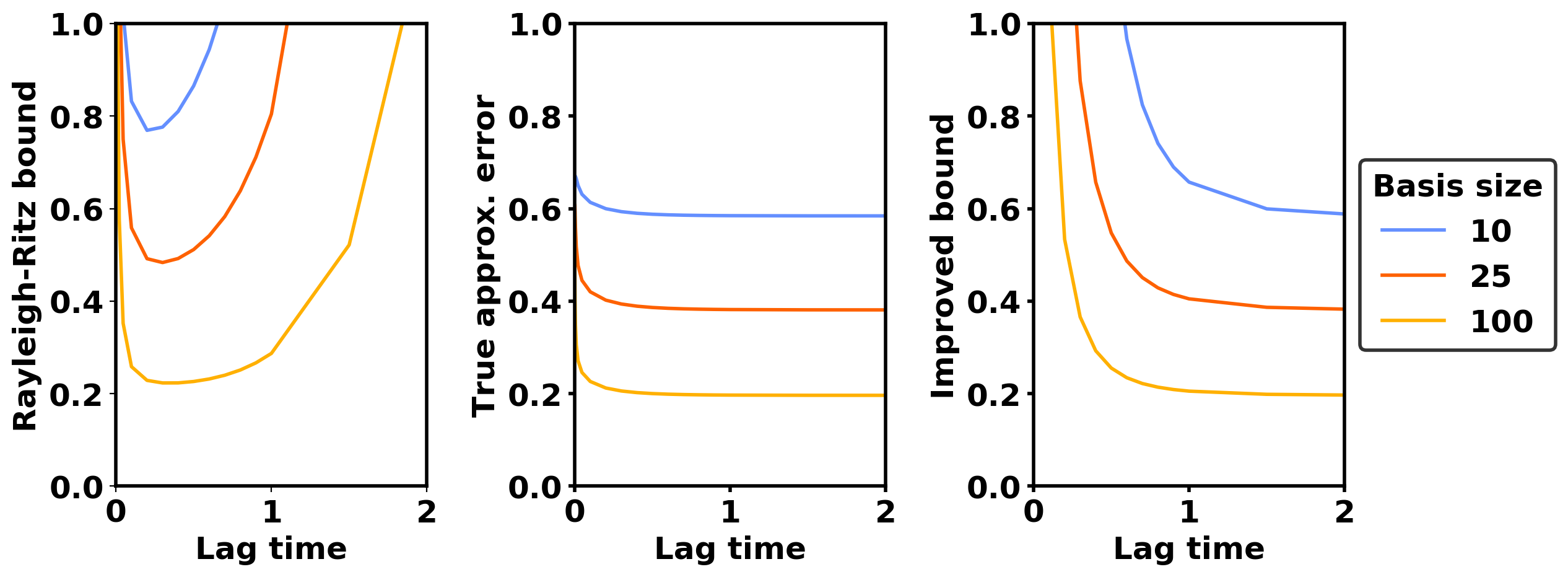}
    \captionsetup{singlelinecheck=off}
    \caption[.]{Comparison between different bounds for the the approximation error.
    Left: the Rayleigh-Ritz bound
    asymptotes to infinity at long and short lag times.
    Center: the true approximation error 
    stabilizes at long lag times.
    Right: the improved bound presented in \cref{thm:remarkable} is asymptotically sharp at long lag times.}
    \label{fig:stabilizes}
\end{figure}


\subsubsection{New approximation bounds}{\label{subsub:new}}
To analyze the dependence on lag time,
we develop a mathematical approach different from the methods applied to the Rayleigh-Ritz method in the past.
We start by identifying a key stability property of idealized VAC
that has not appeared in the previous literature.
As $\tau \rightarrow \infty$, 
idealized VAC eigenspaces converge
to a well-defined limit,
implying that the approximation error
must stabilize at long lag times.

To rigorously study the convergence of idealized VAC estimates,
our first step is to introduce the orthogonalized projection functions 
$q_1, q_2, \ldots$.
These are the natural functions to appear in the $\tau \rightarrow \infty$ limit.
They are constructed from the projected eigenfunctions $P_{\Phi} \eta_1, P_{\Phi} \eta_2, \ldots,$
but they are adjusted
to meet the orthogonality constraints on idealized VAC eigenfunctions.

\begin{definition}{\label{def:orthogonal}}
Set $p = \min\left\{n, r\right\}$,
where $n$ is the number of basis functions $\left(\phi_i\right)_{1 \leq i \leq n}$
and $r$ is the number of eigenfunctions $\left(\eta_i\right)_{1 \leq i \leq r}$ identified in \cref{ass:quasicompact}.
Assume that projections
$P_{\Phi} \eta_i$ are linearly independent for $1 \leq i \leq p$.
Then, define:
\begin{align}
\label{eq:gs_start}
    & \tilde{q}_1 = P_{\Phi} \eta_1,
    & q_1 = \tilde{q}_1 \slash \left\lVert \tilde{q}_1 \right\rVert, \\
    & \tilde{q}_2 = P_{\Phi} \eta_2 - \left<q_1, \eta_2\right> q_1,
    & q_2 = \tilde{q}_2 \slash \left\lVert \tilde{q}_2 \right\rVert, \\
    & \vdots
    & \\
    \label{eq:gs_end}
    & \tilde{q}_p = P_{\Phi} \eta_p 
    - \sum_{i=1}^{p-1} \left<q_i, \eta_p\right> q_i,
    & q_p = \tilde{q}_p \slash \left\lVert \tilde{q}_p \right\rVert.
\end{align}
\end{definition}

Our next step is to prove that idealized VAC eigenfunctions $\gamma_i^\tau$
converge to the orthogonalized projections $q_i$ at long lag times.

\begin{theorem}[The $\tau \rightarrow \infty$ limit]
\label{thm:remarkable}
Fix the basis set $\Phi$ but allow the lag time $\tau$ to vary.
In the limit 
as $\tau \rightarrow \infty$,
idealized VAC estimates converge as follows:
\begin{enumerate}[leftmargin = *]
\item
When there is a gap between $\sigma_k$ and all other $\sigma_i$ values,
the $k$th idealized VAC eigenvalue satisfies
\begin{equation}
\label{eq:refined_eigenvalue}
    \frac{\lambda_k^{\tau}}{e^{-\sigma_k \tau}}
    \rightarrow \left<\eta_k, q_k\right>^2.
\end{equation}
\item{\label{item:remarkable_subspaces}}
When there is a gap between $\left\{\sigma_j, \ldots, \sigma_k\right\}$ and all other $\sigma_i$ values,
the subspace $\spa_{j \leq i \leq k} \gamma_i^{\tau}$ of idealized VAC eigenfunctions converges
\begin{equation}
\label{eq:remarkable_subspaces}
\spa_{j \leq i \leq k} \gamma_i^{\tau}
\rightarrow \spa_{j \leq i \leq k} q_i.
\end{equation}
\item 
When there is a gap between $\sigma_k$ and 
all other $\sigma_i$ values
and a gap between $\sigma_{k+1}$ and all other $\sigma_i$ values,
the top $k$ idealized VAC eigenfunctions satisfy
\begin{equation}
\label{eq:equivalent}
    d_{\F}\left(\spa_{1 \leq i \leq k} \gamma_i^{\tau},
    \spa_{1 \leq i \leq k} q_i\right)\frac{\lambda_k^\tau}
    {\lambda_{k+1}^{\tau}}
    \rightarrow
    \left|\frac{\left<\eta_{k+1}, q_k\right>}
    {\left<\eta_{k+1}, q_{k+1}\right>}\right|.
\end{equation}
\end{enumerate}
\end{theorem}
\begin{proof}
See \cref{sub:eigenvalue_convergence},
\cref{sub:subspace_convergence}, and \cref{sub:error_multiplier}.
\end{proof}

The main message of \cref{thm:remarkable} 
is
that idealized VAC eigenspaces converge exponentially fast
as $\tau \rightarrow \infty$.
Because of this convergence, the approximation error must stabilize.
As the last step of our approximation error analysis,
we use the stabilization at long lag times to provide a new, asymptotically sharp bound on VAC's approximation error.

\begin{theorem}{\label{thm:remarkable2}}
When $\lambda_k^{\tau} > e^{-\sigma_{k+1} \tau}$,
the top $k$ idealized VAC eigenfunctions are bounded by
\begin{equation}
\label{eq:advantageous}
    1 \leq \frac{d_{\F}^2\left(\spa\limits_{1 \leq i \leq k} \gamma_i^{\tau}, \spa\limits_{1 \leq i \leq k} \eta_i \right)}
    {d_{\F}^2\left(\spa\limits_{1 \leq i \leq k} \eta_i, \Phi\right)}
    \leq 1 + \frac{1}{4} \left|\frac{e^{-\sigma_{k+1} \tau}}
    {\lambda_k^{\tau} - e^{-\sigma_{k+1} \tau}}\right|^2.
\end{equation}
\end{theorem}
\begin{proof}
See \cref{sub:subspace_convergence}.
\end{proof}


Interpreting the results of this subsection,
we can identify concrete strategies for how best to reduce approximation error.
The approximation error can be divided into two parts:
\begin{equation}
\label{eq:lower_bound}
    \underbrace{d_{\F}\left(\spa_{j \leq i \leq k} \gamma_i^{\tau},
    \spa_{j \leq i \leq k} \eta_i\right)}_{\text{approximation error}}
    \leq
    \underbrace{d_{\F}\left(\spa_{j \leq i \leq k} q_i,
    \spa_{j \leq i \leq k} \eta_i\right)}_{\text{lag-time-independent error}}
    + \underbrace{d_{\F}\left(\spa_{j \leq i \leq k} \gamma_i^{\tau},
    \spa_{j \leq i \leq k} q_i\right)}_{\text{lag-time-dependent error}}.
\end{equation}
In this decomposition, we separate the lag-time-independent error and the lag-time-dependent
error. In applications of VAC, there are separate strategies for reducing these two error sources.

To reduce the lag-time-independent error,
the best strategy is to enrich the basis set as much as possible.
If the basis set is rich enough to approximate
the top eigenfunctions $\eta_1, \eta_2, \ldots, \eta_k$
with high accuracy,
then the lag-time-independent error must be low.
Assuming there is a gap between $\left\{\sigma_j, \ldots, \sigma_k\right\}$
and all other $\sigma_i$ values,
\cref{lem:quadratic} guarantees
\begin{equation}
\label{eq:right_hand_side}
    \underbrace{d_{\F}^2\left(\spa_{j \leq i \leq k} q_i,
    \spa_{j \leq i \leq k} \eta_i\right)}_{\text{squared lag-time-independent error}}
    \leq d_{\F}^2\left(\spa_{1 \leq i \leq j-1} \eta_i, \Phi \right)
    + d_{\F}^2\left(\spa_{1 \leq i \leq k} \eta_i, \Phi \right).
\end{equation}
As the basis set becomes increasingly rich, the right-hand-side
of the inequality converges to zero,
implying that the lag-time-independent error must vanish.


To reduce the lag-time-dependent error,
the best strategy is simply to increase the lag time.
As $\tau \rightarrow \infty$, 
\cref{thm:remarkable} guarantees that
the lag-time-dependent error
must decay exponentially fast.



\subsection{Estimation error}{\label{sub:estimation}}

In this subsection, we present formulas for the estimation error
and explain how to calculate the mean squared estimation error using data.

\subsubsection{Formulas for the estimation error}

In applications of VAC, it is not typically possible to evaluate expectations
$C_{ij}\left(\tau\right) = \E_{\mu}\left[\phi_i\left(X_0\right) \phi_j\left(X_\tau\right)\right]$ exactly.
Instead, trajectory data is used to provide estimates
$\hat{C}_{ij}\left(\tau\right) \approx C_{ij}\left(\tau\right)$.
In the asymptotic limit as $\hat{C}\left(\tau\right) \rightarrow C\left(\tau\right)$
and $\hat{C}\left(0\right) \rightarrow C\left(0\right)$,
the estimation error is governed by the following asymptotic formulas.

\begin{theorem}[Estimation error]{\label{thm:samplingError}}
Fix the basis set $\Phi$ and the lag time $\tau > 0$,
but allow matrices $\hat{C}\left(0\right)$
and $\hat{C}\left(\tau\right)$ to vary.
Assume idealized VAC eigenfunctions are normalized so that $\left<\gamma_i^{\tau}, \gamma_j^{\tau}\right> = \delta_{ij}$,
and recall $v_i\left(\tau\right)$
is the vector with
$\gamma_i^\tau = \sum_{j=1}^n v_j^i\left(\tau\right) \phi_j$.
Set
\begin{equation}
\hat{L}_{ij}\left(\tau\right) = v^i\left(\tau\right)^T
\left[\hat{C}\left(\tau\right) 
- \lambda_j^{\tau} \hat{C}\left(0\right)
\right]
v^j\left(\tau\right),
\quad
1 \leq i,j \leq n.
\end{equation}
Then, VAC estimates have the following behavior as $\hat{C}\left(\tau\right) \rightarrow C\left(\tau\right)$ and $\hat{C}\left(0\right) \rightarrow C\left(0\right)$:
\begin{enumerate}[leftmargin = *]
\item
When there is a gap between $\lambda_k^{\tau}$
and all other $\lambda_i^{\tau}$ values,
the $k$th VAC eigenvalue satisfies
\begin{equation}
\label{eq:eigenvalue_error}
    \hat{\lambda}_k^{\tau} - \lambda_k^{\tau}
    =  \hat{L}_{kk}\left(\tau\right)
    + \mathcal{O}\left(\left\lVert \hat{C}\left(\tau\right) - C\left(\tau\right)\right\rVert^2_{\F}
    + \left\lVert \hat{C}\left(0\right) - C\left(0\right)\right\rVert^2_{\F}\right).
\end{equation}
\item
When there is a gap between
$\left\{\lambda_j^{\tau}, \ldots, \lambda_k^{\tau}\right\}$
and all other $\lambda_i^{\tau}$ values,
the subspace $\spa_{j \leq i \leq k} \hat{\gamma}_i^{\tau}$
of VAC eigenfunctions satisfies
\begin{align}
\label{eq:asymptotic}
\begin{aligned}
    d_{\F}\left(\spa_{j \leq i \leq k} \hat{\gamma}_i^{\tau}, \spa_{j \leq i \leq k} \gamma_i^{\tau}\right)
    & = \left(\sum_{\substack{l < j \\ \text{or } l > k}}
    \sum_{m = j}^k
    \left|\frac{\hat{L}_{l m}^{\tau}}
    {\lambda_l^\tau - \lambda_m^{\tau}}\right|^2\right)^{1 \slash 2} \\
    & + \mathcal{O}\left(\left\lVert \hat{C}\left(\tau\right) - C\left(\tau\right)\right\rVert^2_{\F}
    + \left\lVert \hat{C}\left(0\right) - C\left(0\right)\right\rVert^2_{\F}\right).
\end{aligned}
\end{align}
Moreover, the condition number for the subspace $\spa_{j \leq i \leq k} \hat{\gamma}_i^{\tau}$ is given by
\begin{equation}
\label{eq:condition}
    \limsup_{\substack{\hat{C}\left(\tau\right) \rightarrow C\left(\tau\right) \\ \hat{C}\left(0\right) \rightarrow C\left(0\right)}}
    \frac{d_{\F}\left(\spa\limits_{j \leq i \leq k} \hat{\gamma}_i^{\tau}, \spa\limits_{j \leq i \leq k} \gamma_i^{\tau}\right)} {\left\lVert 
    \hat{L}\left(\tau\right)
    \right\rVert_{\F}} = \frac{1}{\min\left\{\lambda_{j-1}^{\tau} - \lambda_j^{\tau},
    \lambda_k^{\tau} - \lambda_{k+1}^{\tau}\right\}},
\end{equation}
with the conventions $\lambda_0^{\tau} = \infty$ and $\lambda_{n+1}^{\tau} = -\infty$.
\end{enumerate}
\end{theorem}
\begin{proof}
See \cref{sub:matrix_perturbation}.
\end{proof}


A useful quantity identified in \cref{thm:samplingError} is the condition number \eqref{eq:condition},
which quantifies VAC’s sensitivity to small errors in the matrices $\hat{C}\left(\tau\right)$ and $\hat{C}\left(0\right)$.
In experiments, we find the condition number is a useful heuristic for judging whether a VAC estimation
problem is easy or hard—more specifically whether a large or small data set is required for
accurate estimation.
When the condition number 
for a subspace of VAC eigenfunctions is higher than $5$ at all lag times, 
the numerical experiments in \cref{sec:experiments} show that
VAC is prone to experiencing large amounts of estimation error.
Empirically, we can estimate the minimum condition number across all lag times by using
\begin{equation}
\min_{\tau \geq 0} \frac{1}{\min\left\{\hat{\lambda}_{j-1}^\tau - \hat{\lambda}_j^\tau,
\hat{\lambda}_k^\tau - \hat{\lambda}_{k+1}^\tau\right\}}.
\end{equation}
We recommend that VAC users identify the minimum condition number for various subspaces and focus on estimating the well-conditioned
subspaces whenever possible.
Additionally, we recommend that authors report the minimum condition number along with their VAC results, helping readers to assess whether the results 
could be affected by estimation error.


\subsubsection{Calculating the asymptotic mean squared estimation error using data}

Here, we explain how to calculate the mean squared estimation error using trajectory data.
We assume for simplicity that
the data consists of a
single long stationary trajectory
of the process $X_t$. However, the estimation procedure described here could be generalized to other types of trajectory data.

Our approach for calculating the mean squared estimation error is based on the following convergence in distribution result.

\begin{theorem}{\label{thm:moments}}
Fix the basis set $\Phi$ and the lag time $\tau > 0$, but allow the data set used in VAC to vary.
Assume $\E_{\mu}\left|\phi_i\left(X_0\right)\right|^4 < \infty$ for $1 \leq i \leq n$.
Assume a stationary trajectory $ X_0, X_{\Delta}, X_{2 \Delta}, \ldots, X_{T-\Delta}$
is simulated and estimates $\hat{C}_{ij}\left(t\right) \approx C_{ij}\left(t\right)$ are formed using
\begin{equation}
    \hat{C}_{ij}\left(t\right)
    = \frac{\Delta}{T - t} 
    \sum_{s=0}^{\frac{T - t}{\Delta} - 1}
    \frac{\phi_i\left(X_{s \Delta}\right) \phi_j\left(X_{s \Delta + t}\right)
    + \phi_j\left(X_{s \Delta}\right) \phi_i\left(X_{s \Delta + t}\right)}{2}.
\end{equation}
Then,
VAC estimates have the following behavior
as $T \rightarrow \infty$:
\begin{enumerate}[leftmargin = *]
    \item 
    When there is a gap between $\lambda_k^{\tau}$
    and all other $\lambda_i^{\tau}$ values, the $k$th VAC eigenvalue satisfies
    \begin{equation}
    \label{eq:formula1}
    \sqrt{T}
    \left(\hat{\lambda}_k^{\tau} - \lambda_k^{\tau}\right)
    \stackrel{\mathcal{D}}{\rightarrow}
    Z_{kk}^{\tau}.
    \end{equation}
    \item
    When there is a gap between
    $\left\{\lambda_j^{\tau}, \ldots, \lambda_k^{\tau}\right\}$
    and all other $\lambda_i^{\tau}$ values,
    the subspace $\spa_{j \leq i \leq k} \hat{\gamma}_i^{\tau}$
    of VAC eigenfunctions satisfies
    \begin{equation}
    \label{eq:formula2}
    T d_{\F}\left(\spa_{j \leq i \leq k} \hat{\gamma}_i^{\tau}, \spa_{j \leq i \leq k} \gamma_i^{\tau}\right)^2
    \stackrel{\mathcal{D}}{\rightarrow} \sum_{\substack{l < j \\ \text{or } l > k}}
    \sum_{m = j}^k
    \left|\frac{Z_{l m}^{\tau}}
    {\lambda_l^\tau - \lambda_m^{\tau}}\right|^2.
    \end{equation}
\end{enumerate}
Here, $\left(Z_{lm}^{\tau}\right)_{1 \leq l,m \leq n}$ is a mean-zero multivariate Gaussian with variance terms
\begin{equation}
\label{eq:asym_var}
    \E \left|Z_{l m}^{\tau}\right|^2
    = \Delta
    \sum_{s = -\infty}^{\infty}
    \Cov_{\mu}\left[F_{l m}^\tau\left(X_0, X_\tau\right),
    F_{l m}^\tau\left(X_{s \Delta}, X_{s \Delta + \tau}\right)\right],
\end{equation}
where we have defined
\begin{equation}
    F_{l m}^\tau\left(x, y\right) =
    \frac{\gamma_{l}^\tau \left(x\right)
    \gamma_m^\tau\left(y\right) +
    \gamma_{l}^\tau\left(y\right)
    \gamma_m^\tau\left(x\right)}{2}
    - \lambda_m^\tau \frac{
    \gamma_{l}^\tau\left(x\right)
    \gamma_m^\tau\left(x\right) +
    \gamma_{l}^\tau\left(y\right)
    \gamma_m^\tau\left(y\right)}{2}.
\end{equation}
\end{theorem}
\begin{proof}
See \cref{sub:variance}
\end{proof}

The great value of \cref{thm:moments} is that
it suggests a data-driven approach for calculating 
the mean squared estimation error in the asymptotic limit.
First, we can use the data set to estimate the $\E\left|Z_{lm}^{\tau}\right|^2$ terms by means of equation \eqref{eq:asym_var}.
Second, we can substitute the $\E\left|Z_{lm}^{\tau}\right|^2$ estimates into equation
\eqref{eq:formula1} or \eqref{eq:formula2} to compute the mean squared estimation error for eigenvalues or invariant subspaces.
In the supplement, we provide a step-by-step description of this estimation procedure.


In \cref{fig:error_scaling}, we calculate the mean squared estimation error by using a single trajectory of data.
We find that it is possible to accurately identify the lag times 
at which the mean squared estimation error exceeds a critical threshold, such as $0.2$.
Moreover, in the numerical experiments in \cref{sec:experiments}, $0.2$ is a typical level at which
the estimation error begins to contribute significantly to VAC's overall error.
Therefore, in VAC applications we recommend calculating the mean squared error and avoiding lag times where the error exceeds such a threshold.


\begin{figure}[h!]
    \centering
    \includegraphics[scale = 0.4, clip]{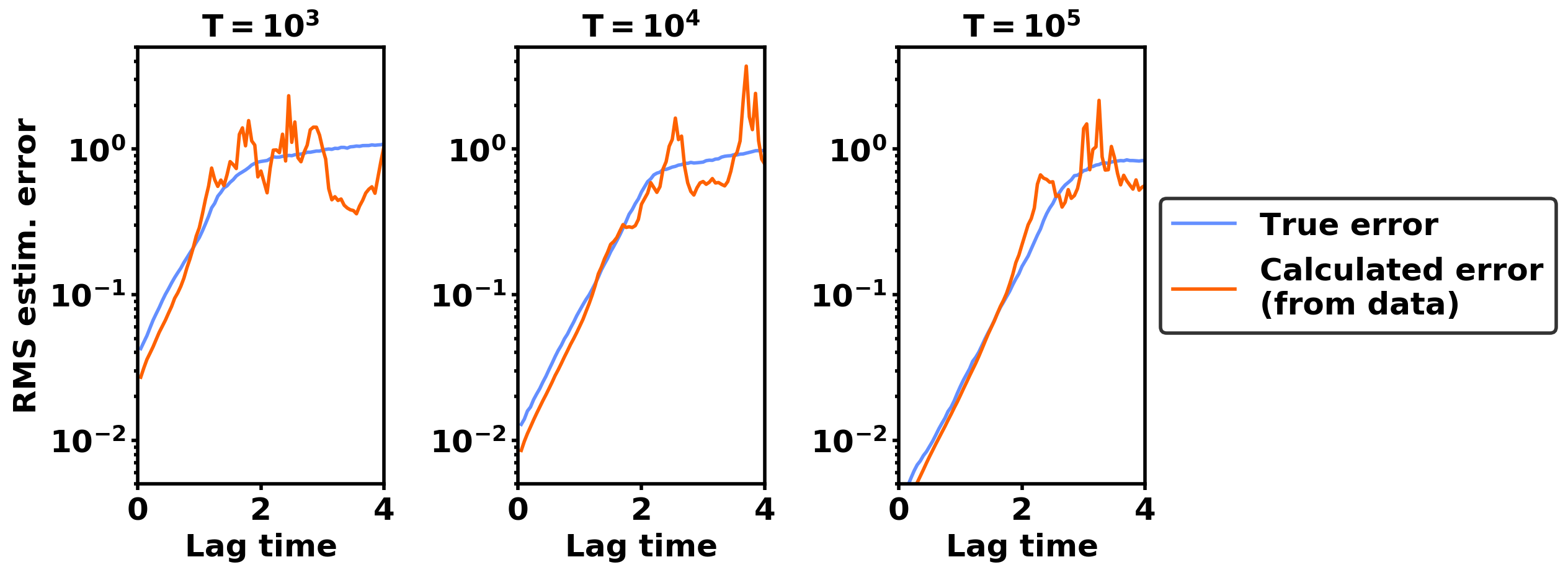}
    \caption{Root mean squared estimation error for different trajectory lengths $T$.
    The calculated error is obtained from a single trajectory of data using formulas
    in \cref{thm:moments}.
    The true error is obtained through $100$ independent trials.
    \label{fig:error_scaling}}
\end{figure}

We conclude this section by considering three additional strategies to reduce the estimation error of VAC.
The first strategy is to increase trajectory length.
By increasing the length $T$ of the trajectory,
the estimation error consistently decreases at a  $1 \slash \sqrt{T}$ rate
as shown in \cref{fig:error_scaling}.

The second strategy for reducing the estimation error is to prune the size of the basis set.
We find in \cref{thm:samplingError}
that the squared estimation error
increases linearly with the number of basis functions.
Therefore, it is best to include only those basis functions that have the potential to overlap
with the eigenfunctions of the transition operator.

The final strategy for reducing the estimation error is to select basis functions
with favorable integrability properties.
In \cref{thm:moments}, it is seen that the mean squared estimation error depends on the fourth moments of the idealized VAC coordinates.
If the basis functions themselves have large kurtosis
\begin{equation}
    \frac{\E_{\mu}\left|\phi_i\left(X_0\right) - \E_{\mu}\left[\phi_i\left(X_0\right)\right]\right|^4}
    {\left(\E_{\mu}\left|\phi_i\left(X_0\right) - \E_{\mu}\left[\phi_i\left(X_0\right)\right]\right|^2\right)^2},
\end{equation}
this can increase the estimation error in VAC calculations.
Favorable integrability properties may be one factor that helps explain the success of Markov state models, in which the basis consists of indicator functions on a partition of the state space.
The higher moments of indicator functions are often well-controlled, 
compared to, e.g., higher order polynomials of the coordinate axes.


\section{Numerical experiments}
\label{sec:experiments}

In this section, we report on two numerical experiments 
that illustrate the major factors impacting VAC accuracy.
Moreover, these experiments show how computing the VAC condition number and the mean squared estimation error
can help to improve VAC's accuracy.



\subsection{Varying the basis size and trajectory length}

First, we apply VAC to estimate the span of eigenfunctions $\eta_1$, $\eta_2$ and $\eta_3$
for the Ornstein-Uhlenbeck (OU) process
\begin{equation}
    \mathop{dX} = -X \mathop{dt} + \sqrt{2} \mathop{dW}.
\end{equation}
In two different trials,
we show how VAC's accuracy depends on the size of the basis set and the length of the simulated trajectory.
The number of basis functions and the trajectory length are varied as follows:
\vspace{.2cm}
\begin{center}
\begin{tabular}{l|l|l}
        ~ & Trial $1$ & Trial $2$  \\
        \hline
        Basis functions & $n = 20$ & $n = 50$ \\
        Trajectory length & $T = 10^4$ & $T = 500$
\end{tabular}
\end{center}
\vspace{.2cm}
In both trials, the basis functions are indicator functions on disjoint intervals.

The two different trials demonstrate that the breakdown of approximation error and estimation error is sensitive to context, as seen in \cref{fig:OU_example}.
The approximation error is higher in trial 1
because of the smaller set of basis functions,
whereas the estimation error is higher in trial 2
because of the smaller data set.
In trial 1, it is optimal to use a comparatively long lag time of $\tau = 0.7$ to reduce the approximation error.
In contrast, in trial 2 it is
optimal to use a comparatively short lag time of $\tau = 0.1$
to avoid the increase in estimation error at longer lag times.

\begin{figure}[h!]
    \centering
    \includegraphics[scale = .35, clip]{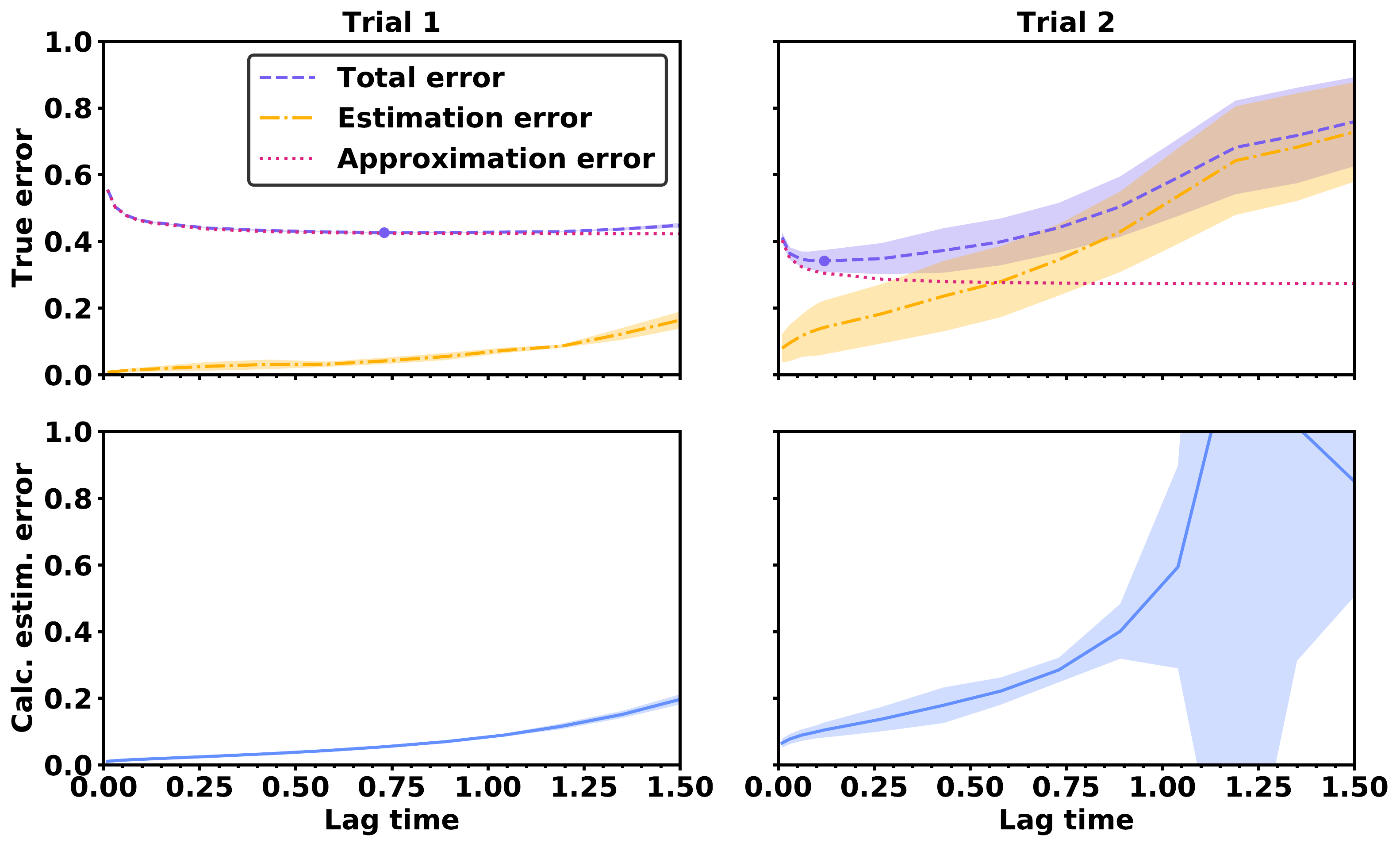}
    \caption{
    True error and calculated error for the OU process.
    Top: the bold line indicates the true root mean squared (RMS) error over $30$ independent trajectories,
    while the shaded region indicates the mean $\pm$ $1$ standard deviation.
    The purple dot shows the optimal lag time.
    Bottom: the calculated RMS estimation error
    obtained from each of the $30$ trajectories.
    \label{fig:OU_example}}
\end{figure}

In addition to showing the true error levels, 
\cref{fig:OU_example} shows the root mean squared estimation error calculated directly from the data.
In trial 1, the calculated estimation error remains below $0.2$ for lag times up to $1.5$.
Therefore, the VAC practitioner can infer that the 
data set is rich enough to take high lag times
without experiencing large estimation error.
However, in trial 2,
the calculated estimation rises more rapidly, 
reaching a level of $0.2$ when the lag time is just $\tau \approx 0.5$.
In this case, the VAC practitioner can infer that the 
data set is not rich enough to take $\tau > 0.5$.


An alternative lag time selection strategy
called \emph{implied timescale analysis}
has been advocated in the past by VAC researchers \cite{swope2004describing}.
In this strategy, the VAC eigenvalues are used to compute the implied timescales
\begin{equation}
    \frac{-\tau}{\log\left(\hat{\lambda}_i^{\tau}\right)}.
\end{equation}
If VAC eigenvalues are perfect estimates of the true eigenvalues, then implied timescales are perfectly flat
and they equal $\sigma_i^{-1}$.
In practice, however,
implied timescales are not flat.
They
increase quickly at short lag times and then increase more slowly at long lag times.
To cut down on VAC's error,
Swope and coauthors \cite{swope2004describing}
proposed selecting a long enough lag time
so that the implied timescales for the eigenfunctions of interest
are approximately level.

\Cref{fig:implied_timescales} presents the implied timescales for the OU process.
From the figure it is clear that the implied timescales cannot be used to assess the estimation error.
The estimation
error is much higher in the second trial, yet the implied timescales for trial 1 and trial 2 are
similar.
However, implied timescales may help assess the approximation error.
As the lag time is increased from $\tau = 0$
to $\tau = 0.1$, 
the second and third implied timescale become much flatter,
which provides an accurate indication that
the approximation error is decreasing and beginning to settle.

\begin{figure}[h!]
    \centering
    \includegraphics[scale = .4, clip]{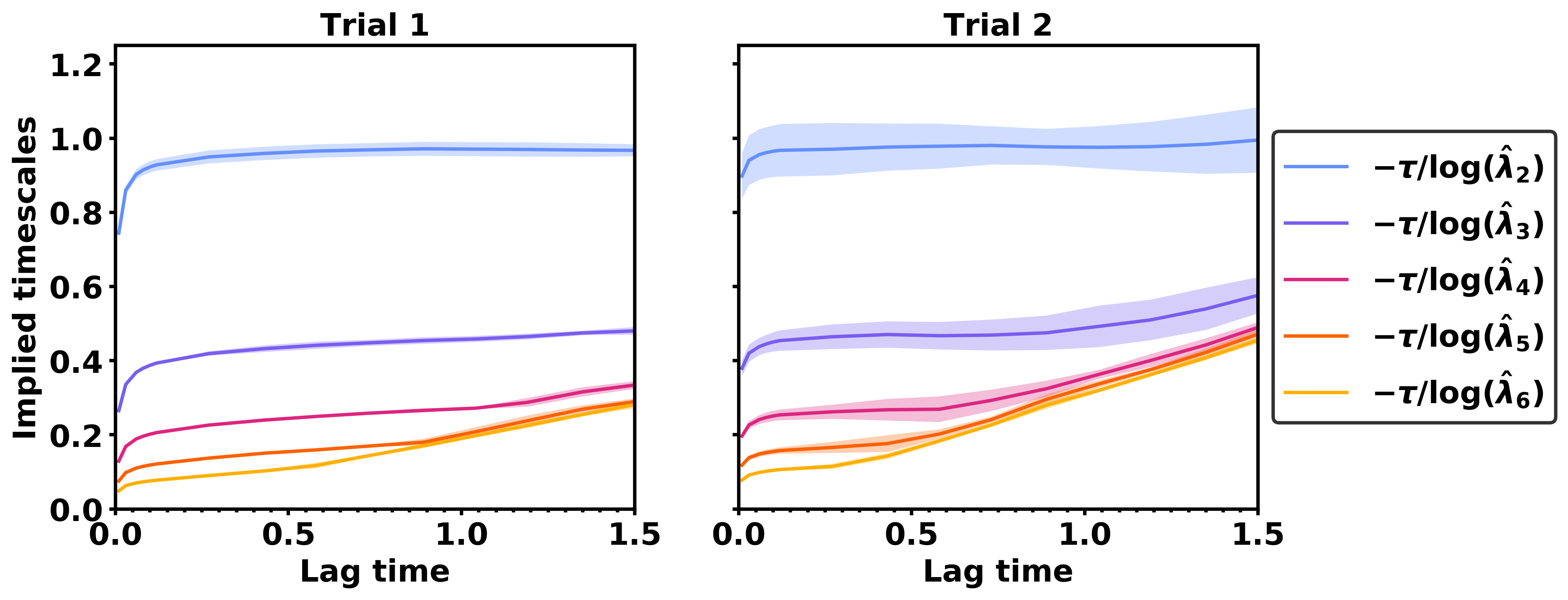}
    \caption{Implied timescales of the OU process.}
    \label{fig:implied_timescales}
\end{figure}

We conclude that implied timescale analysis may provide an approach for assessing approximation error that is complementary to our approach for assessing the estimation error.
Whereas our approach is useful for identifying and avoiding the error that is prevalent at long lag times,
implied timescale analysis may be useful for identifying and avoiding the error that is prevalent at short lag times.
However, while our approach for computing the mean squared estimation error is rigorously justified,
it remains an open research problem
to rigorously justify this proposed relationship between the implied timescales and the approximation error.



\subsection{Varying the size of the subspace}

In this second experiment, we apply VAC to estimate the eigenfunctions
of the diffusion process
\begin{equation}
    \mathop{dX} = - \frac{1}{2}\sigma \sigma^T \nabla U(X) \mathop{dt} + \sigma \mathop{dW},
\end{equation}
where the potential $U$ and the diffusion matrix 
$\sigma$ are given by:
\begin{align}
    & U(x_1, x_2) = 4 x_1^4 - 8 x_1^2 +  x_1 + 0.5 x_2^2,
    & \sigma = 
    \begin{pmatrix} 2 & 0 \\ -1 & \sqrt{3} \end{pmatrix}.
\end{align}
We simulate a stationary trajectory of length $T = 500$
and then apply VAC using the basis set
$\left\{1, x_1, x_2, x_1^2, x_1 x_2, x_2^2\right\}$.

We investigate how the accuracy changes
when VAC is used to estimate two subspaces of different sizes:
$\spa\left\{\eta_1, \eta_2\right\}$
and $\spa\left\{\eta_1, \eta_2, \eta_3\right\}$.
When estimating $\spa\left\{\eta_1, \eta_2\right\}$,
there is a wide range of lag times that 
all lead to low error levels. As seen in \cref{fig:double_well},
the total error decreases between lag times of $\tau = 0$ and $\tau = 0.2$, but it is nearly constant for all lag times between $\tau = 0.2$ to $\tau = 1.5$.
In contrast, when estimating $\spa\left\{\eta_1, \eta_2, \eta_3\right\}$,
the total error is V-shaped with a distinct minimum at the lag time $\tau = 0.2$.
The error rises rapidly as the lag time is increased beyond $\tau = 0.2$ due to an upsurge in the estimation error.

\begin{figure}[h!]
    \centering
    \includegraphics[scale = .35, clip]{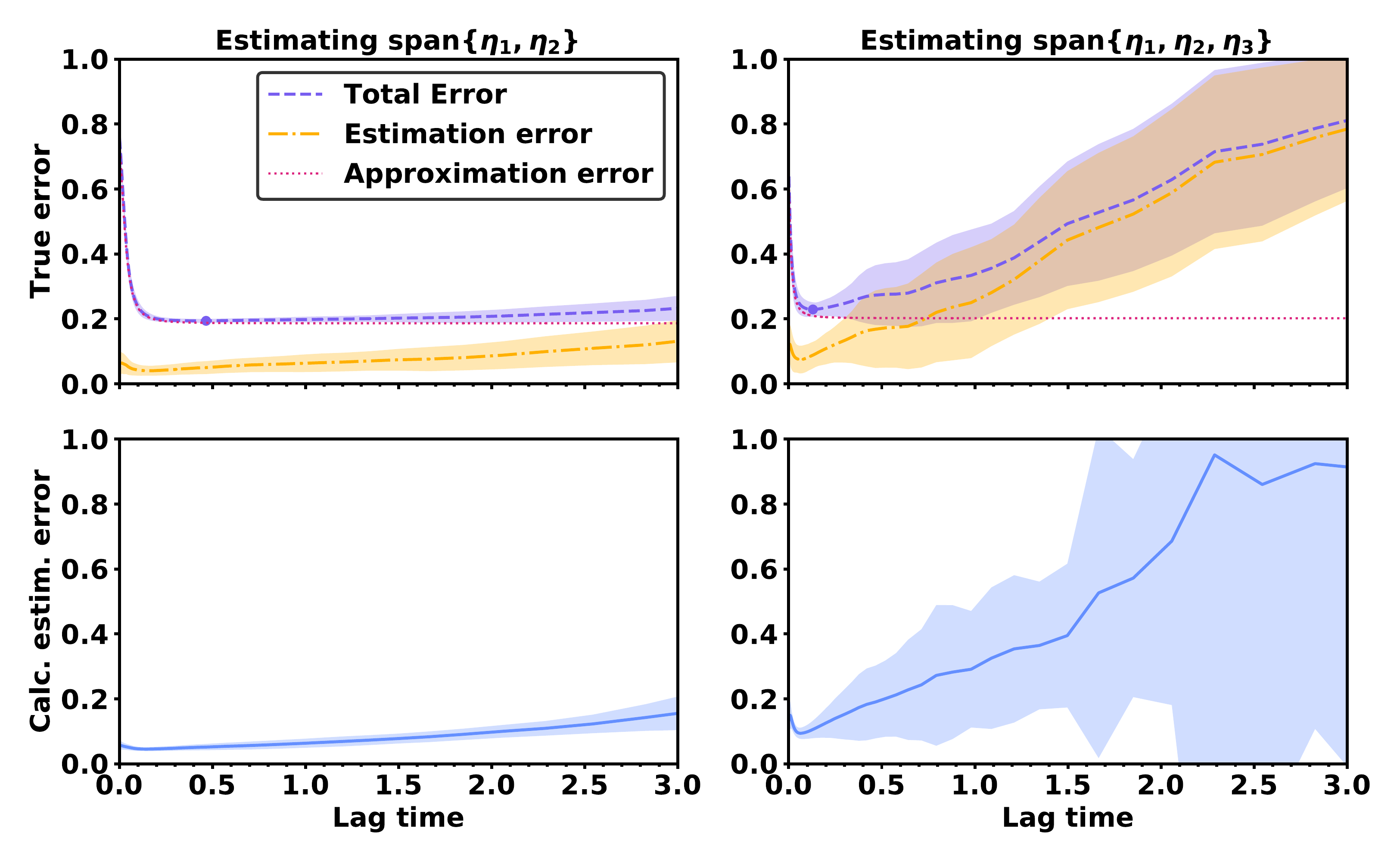}
    \caption{When the minimum condition number is $2.0$, estimation error is low (left). 
    When the minimum condition number is $9.5$, estimation error is much higher (right).}
    \label{fig:double_well}
\end{figure}


What explains the different error profiles when estimating the subspace $\spa\left\{\eta_1, \eta_2\right\}$ versus $\spa\left\{\eta_1, \eta_2, \eta_3\right\}?$
The explanation is not a difference in the data set or the basis set,
since these factors remain the same when estimating the two subspaces.
Rather, the increase in estimation error
is due to the much higher condition number for the subspace $\spa\left\{\eta_1, \eta_2, \eta_3\right\}$.
No matter how the lag time is selected,
the inverse spectral gap $\left(\lambda_4^\tau - \lambda_3^\tau\right)^{-1}$ is at least as high as $9.5$.
In contrast, when estimating the subspace $\spa\left\{\eta_1, \eta_2\right\}$,
the minimum condition number $\min_{\tau} \left(\lambda_3^\tau - \lambda_2^\tau\right)^{-1}$ is just $2.0$.
Here we see a high condition number is associated with increased levels of estimation error
and a stronger relationship between estimation error
and lag time.

To avoid situations where the estimation error is uncontrollably high,
VAC users should identify well-conditioned subspaces
and focus on estimating these subspaces whenever possible.
As shown in the VAC eigenvalue plot in 
\cref{fig:double_well_vac_evals},
eigenvalues for well-conditioned subspaces 
often visually stand apart from the rest of the eigenvalues.
The large gap between the second and third VAC 
eigenvalue indicates a natural separation in timescales,
which implies that $\spa\left\{\eta_1, \eta_2\right\}$ is a well-conditioned subspace.

\begin{figure}[h!]
    \centering
    \includegraphics[scale=0.4, clip]{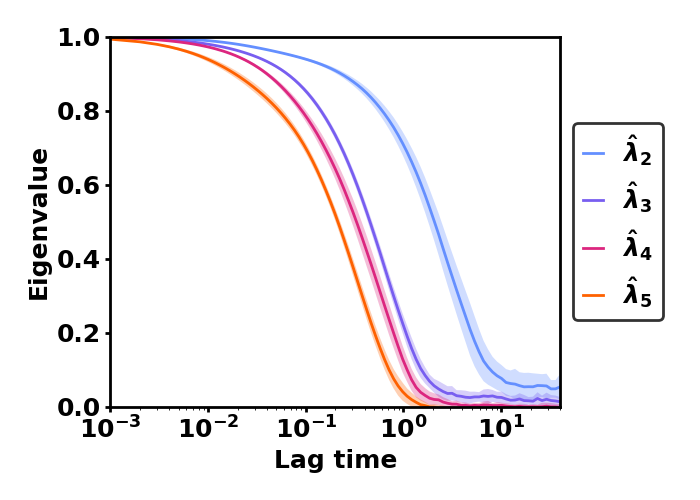}
    \caption{VAC eigenvalues.} 
    \label{fig:double_well_vac_evals}
\end{figure}


\section{Mathematical derivations}\label{sec:math}

In this section, we prove the mathematical results
presented in \cref{thm:remarkable}, \cref{thm:remarkable2}, \cref{thm:samplingError}, and \cref{thm:moments}.


\subsection{Building mathematical intuition}{\label{sub:intuition}}

Before proving \cref{thm:remarkable},
we identify the intuitive mathematical reason
why idealized VAC estimates converge at long lag times.
The intuitive reason for the convergence is revealed through a decomposition of the matrix $C\left(\tau\right)$.
Applying the spectral decomposition \eqref{eq:modified_spectral},
we find that each matrix entry $C_{ij}\left(\tau\right)$
has an exponentially decaying structure.
\begin{equation}
    C_{ij}\left(\tau\right) 
    = \left<\phi_i, T_\tau \phi_j\right>
    = \sum_{l=1}^r e^{-\sigma_l \tau}
    \left<\eta_l, \phi_i\right>
    \left<\eta_l, \phi_j\right> 
    + \mathcal{O}\left(e^{-\sigma_{r+1} \tau}\right),
    \quad \tau \rightarrow \infty.
\end{equation}
Thus, the matrix $C\left(\tau\right)$ is the sum of exponentially decaying rank-one matrices
\begin{align}
\label{eq:expansion}
    C\left(\tau\right) = \sum_{l = 1}^r e^{-\sigma_l \tau} \left<\eta_{l}, \vec{\phi}\right> \left<\eta_{l}, \vec{\phi}\right>^T
    + \mathcal{O}\left(e^{-\sigma_{r+1} \tau}\right),
    \quad \tau \rightarrow \infty,
\end{align}
where we have used the shorthand $\left<\eta_{l}, \vec{\phi}\right>$ 
to denote the vector with entries $\left<\eta_l, \phi_i\right>$.

To approximate the behavior of idealized VAC at long lag times,
we remove
the smallest terms in
the expansion \eqref{eq:expansion}
and replace $C\left(\tau\right)$ 
by the sum of $k$ rank-one matrices
\begin{equation}
    \sum_{l = 1}^k e^{-\sigma_l \tau} \left<\eta_{l}, \vec{\phi}\right> \left<\eta_{l}, \vec{\phi}\right>^T.
\end{equation}
When the rank-$k$ approximation 
is used in place of $C\left(\tau\right)$,
it results that the top $k$ idealized VAC eigenfunctions $\gamma_1^\tau, \ldots, \gamma_k^{\tau}$
span the subspace
\begin{equation}
\spa_{1 \leq i \leq k} q_i
= P_{\Phi} \spa_{1 \leq i \leq k} \eta_i.
\end{equation}
Therefore, the truncation argument helps intuitively explain the convergence of
idealized eigenfunctions $\gamma_i^\tau$ 
to orthogonalized projections $q_i$ as $\tau \rightarrow \infty$.
Our proofs in \cref{sub:eigenvalue_convergence},
\cref{sub:subspace_convergence}, and \cref{sub:error_multiplier}
essentially serve to justify the truncation argument
and to provide rigorous bounds on the convergence.


\subsection{Convergence of eigenvalues}{\label{sub:eigenvalue_convergence}}

In this section, we verify the statement in
\cref{thm:remarkable} that the $k$th idealized eigenvalue converges
\begin{equation}
\tag{\ref{eq:refined_eigenvalue}}
    \frac{\lambda_k^{\tau}}
    {e^{-\sigma_k \tau}} \rightarrow 
    \left<\eta_k, q_k\right>^2
\end{equation}
in the limit $\tau \rightarrow \infty$,
provided there is a gap between $\sigma_k$
and all other $\sigma_i$ values.
To prove this result, our main tool is the min-max principle for self-adjoint operators \cite{reed1978methods}.
\begin{lemma}{\label{lem:min-max}}
Consider a quasi-compact self-adjoint operator
\begin{equation}
    A = \sum\nolimits_{i=1}^r \lambda_i P_{\eta_i} + R.
\end{equation}
Here, $\eta_1, \eta_2, \ldots, \eta_r$ are orthonormal eigenfunctions of $A$
with eigenvalues
$\lambda_1 \geq \lambda_2 \geq \cdots \geq \lambda_r$,
and the spectrum of $R$ lies in $\left(-\infty, \lambda_r\right)$.
Then, for each $1 \leq k \leq r$, 
\begin{equation}
    \lambda_k\left(A\right) = \max_{\dim\left(\Eta\right) = k}
    \min_{\eta \in \Eta} \frac{\left<\eta, A \eta\right>}{\left<\eta, \eta\right>}
\end{equation}
\end{lemma}

Before applying the min-max principle, we derive two estimates.

\begin{proposition}{\label{prop:upper_bound}}
For any $\phi \in \Phi \cap 
\left(\spa_{1 \leq i \leq k-1} q_i\right)^{\perp}$,
\begin{equation}
\frac{\left<\phi, T_{\tau} \phi\right>}
{\left<\phi, \phi\right>}
\leq e^{-\sigma_k \tau} \left<\eta_k, q_k\right>^2 + e^{-\sigma_{k+1} \tau}.    
\end{equation}
\end{proposition}
\begin{proof}
Calculate
\begin{align}
\left<\phi, T_{\tau}\phi\right>
&= \left<\phi, \left(\sum\nolimits_{i = k}^r e^{-\sigma_i \tau} P_{\eta_i} + R_{\tau}\right) \phi\right> \\
&= e^{-\sigma_k \tau} \left<\eta_k, \phi\right>^2 + 
\left<\phi,
\left(\sum\nolimits_{i = k + 1}^r e^{-\sigma_i \tau} P_{\eta_i} + R_{\tau}\right) 
\phi\right> \\
&\leq e^{-\sigma_k \tau} 
\left<\eta_k, q_k\right>^2
\left<q_k, \phi\right>^2
+ e^{-\sigma_{k+1} \tau} \left<\phi, \phi\right> \\
&\leq e^{-\sigma_k \tau} \left<\eta_k, q_k\right>^2 \left<\phi, \phi\right>
+ e^{-\sigma_{k+1} \tau} \left<\phi, \phi\right>.
\end{align}
\end{proof}

\begin{proposition}{\label{prop:lower_bound}}
Set $\Eta_{1:k-1} = \spa_{1 \leq i \leq k-1} \eta_i$.
Then for any $q \in Q_{1:k} = \spa_{1 \leq i \leq k} q_i$,
\begin{equation} 
    \frac{\left<q, T_{\tau} q\right>}{\left<q, q\right>} \geq
    e^{-\sigma_k \tau} \left(
    \left<\eta_k, q_k\right>^2
    - \frac{1}
    {e^{\left(\sigma_k -\sigma_{k-1}\right) \tau}
    \left(1 - d_2^2\left(\Eta_{1:k-1}, \Phi\right)\right)
    - \left<\eta_k, q_k\right>^2}
    \right),
\end{equation}
provided the denominator term is positive.
\end{proposition}
\begin{proof}
It suffices to consider the $\left\lVert q \right\rVert = 1$ case.
Then, $q$
can be decomposed as 
$q = a q^\prime + b q_k$
where 
$a^2 + b^2 = 1$,
$q^{\prime} \in Q_{1:{k-1}}$, and
$\left\lVert q^{\prime} \right\rVert = 1$.
It follows
\begin{align}
\left<q, T_{\tau} q\right>
& \geq \left<q, \left(e^{-\sigma_{k-1} \tau} P_{\Eta_{1:k-1}} + e^{-\sigma_k \tau} P_{\eta_k}\right) q\right> \\
& = a^2 e^{-\sigma_{k-1} \tau} \left\lVert P_{\Eta_{1:k-1}} q^{\prime}\right\rVert^2
+ e^{-\sigma_k \tau} \left<\eta_k, a q^\prime + b q_k\right>^2.
\end{align}
Thus, $\left<q, T_{\tau} q\right>$ is bounded from below by the lowest eigenvalue of
\begin{equation}
M = e^{-\sigma_{k-1} \tau} \left\lVert P_{\Eta_{1:k-1}} q^{\prime}\right\rVert^2
\begin{pmatrix}
1 & 0 \\
0 & 0
\end{pmatrix}
+
e^{-\sigma_k \tau}
\begin{pmatrix}
\left<\eta_k, q^\prime\right>^2
& \left<\eta_k, q^{\prime}\right>
\left<\eta_k, q_k\right> \\
\left<\eta_k, q^{\prime}\right> \left<\eta_k, q_k \right>
& \left<\eta_k, q_k\right>^2
\end{pmatrix}.
\end{equation}
For any symmetric real-valued matrix $M = \begin{pmatrix} a & b \\ b & c \end{pmatrix}$ with $a > c$,
the lowest eigenvalue is at least as large
as $c - b^2 \slash \left(a - c\right)$
\cite{mathias1998quadratic}.
We can check that
\begin{equation}
    \left\lVert P_{\Eta_{1:k-1}} q^{\prime}\right\rVert^2
    = 1 - \left\lVert P_{\Eta_{1:k-1}^{\perp}} q^{\prime}\right\rVert^2
    \geq 1 - d_2^2\left(\Eta_{1:k-1}, \Phi\right).
\end{equation}
Therefore, the lowest eigenvalue of the matrix $M$
is at least as large as
\begin{equation}
e^{-\sigma_k \tau} \left<\eta_k, q_k\right>^2
- \frac{e^{-2\sigma_k \tau}}
{e^{-\sigma_{k-1} \tau}
\left(1 - d_2^2\left(\Eta_{1:k-1}, \Phi\right)\right)
- e^{-\sigma_k \tau} \left<\eta_k, q_k\right>^2}.
\end{equation}
\end{proof}

\begin{proof}[Proof of \eqref{eq:refined_eigenvalue}]
Using the min-max principle and \cref{prop:upper_bound},
\begin{equation}
    \lambda_k^{\tau}
    = \max_{\dim\left(\Eta\right) = k, \Eta \subseteq \Phi}
    \min_{\eta \in \Eta} \frac{\left<\eta, T_{\tau} \eta\right>}{\left<\eta, \eta\right>}
    \leq e^{-\sigma_k \tau} \left<\eta_k, q_k\right>^2\left(1 + o\left(1\right)\right)
\end{equation}
as $\tau \rightarrow \infty$.
Using the min-max principle and \cref{prop:lower_bound},
\begin{equation}
    \lambda_k^{\tau} 
    = \max_{\dim\left(\Eta\right) = k, \Eta \subseteq \Phi}
    \min_{\eta \in \Eta} \frac{\left<\eta, T_{\tau} \eta\right>}{\left<\eta, \eta\right>}
    \geq e^{-\sigma_k \tau} \left<\eta_k, q_k\right>^2\left(1 + o\left(1\right)\right).
\end{equation}
\end{proof}


\subsection{Convergence of invariant subspaces}{\label{sub:subspace_convergence}}

In this section, we verify the statement in
\cref{thm:remarkable}
that the subspace $\spa_{j \leq i \leq k} \gamma_i^{\tau}$
of idealized VAC eigenfunctions converges
\begin{equation}
\spa_{j \leq i \leq k} \gamma_i^{\tau}
\rightarrow \spa_{j \leq i \leq k} q_i
\tag{\ref{eq:remarkable_subspaces}}
\end{equation}
in the limit $\tau \rightarrow \infty$,
provided there is a gap between $\left\{\sigma_j, \ldots, \sigma_k\right\}$ and all other $\sigma_i$ values.
To prove this result,
our main tool is a well-known lemma due to Davis and Kahan \cite{davis1970rotation}.

\begin{lemma}{\label{lem:davis}}
Suppose $A$ and $B$ are self-adjoint operators
and $\mathcal{U}$ and $\mathcal{W}$ are closed subspaces.
If the spectrum of $P_{\mathcal{U}} \left.A\right|_{\mathcal{U}}$
lies in the interval $\left[a, b\right]$
and the spectrum of $P_{\mathcal{W}} \left.B\right|_{\mathcal{W}}$
lies in $\left(-\infty, a-\delta\right] \cup \left[b + \delta, \infty\right)$,
\begin{equation}
\label{eq:davis_inequality}
    \delta \left\lVert P_{\mathcal{W}} P_{\mathcal{U}} \right\rVert_{\F} \leq \left\lVert P_{\mathcal{W}} P_{\mathcal{U}} A P_{\mathcal{U}}
    - P_{\mathcal{W}} B P_{\mathcal{W}} P_{\mathcal{U}} \right\rVert_{\F}.
\end{equation}
\end{lemma}

The Davis and Kahan lemma leads
to the following error bound.

\begin{proposition}{\label{prop:mainbound}}
When $\lambda_k^{\tau} > e^{-\sigma_{k+1}}$, the distance between subspaces
$\Gamma_{1:k}^\tau = \spa_{1 \leq i \leq k} \gamma_i^{\tau}$
and
$Q_{1:k} = \spa_{1 \leq i \leq k} q_i$ is bounded by
\begin{equation}
\label{eq:mainbound}
    d_{\F}\left(\Gamma_{1:k}^\tau, Q_{1:k} \right)
    \leq \frac{e^{-\sigma_{k+1} \tau}}
    {2\left(\lambda_k^{\tau} - e^{-\sigma_{k+1} \tau}\right)}
    d_{\F}\left(\spa_{1 \leq i \leq k} \eta_i, \Phi\right).
\end{equation}
\end{proposition}
\begin{proof}
The spectrum of 
$P_{\Phi \cap Q_{1:k}^{\perp}} \left.T_{\tau}\right|_{\Phi \cap Q_{1:k}^{\perp}}$
lies in the interval $\left[0, e^{-\sigma_{k+1} \tau}\right]$,
while the spectrum of $P_{\Gamma_{1:k}^\tau} \left.T_{\tau}\right|_{\Gamma_{1:k}^\tau}$
lies in the interval  $\left[\lambda_k^{\tau}, 1\right]$.
Therefore, the spectral gap is at least $\lambda_k^{\tau} - e^{-\sigma_{k+1} \tau}$.
We calculate
\begin{align}
    \left(\lambda_k^{\tau} - e^{-\sigma_{k+1} \tau}\right)
    \left\lVert
    P_{\Phi \cap Q_{1:k}^{\perp}}
    P_{\Gamma_{1:k}^\tau}
    \right\rVert_{\F}
    & \leq \left\lVert P_{\Phi \cap Q_{1:k}^{\perp}}
    P_{\Gamma_{1:k}^\tau} T_{\tau} P_{\Gamma_{1:k}^\tau}
    - P_{\Phi \cap Q_{1:k}^{\perp}}
    T_{\tau} 
    P_{\Phi \cap Q_{1:k}^{\perp}}
    P_{\Gamma_{1:k}^\tau} \right\rVert_{\F}
    \\
    & = \left\lVert P_{\Phi \cap Q_{1:k}^{\perp}}
    T_{\tau} P_{\Gamma_{1:k}^\tau}
    - P_{\Phi \cap Q_{1:k}^{\perp}}
    T_{\tau} 
    P_{\Phi \cap Q_{1:k}^{\perp}}
    P_{\Gamma_{1:k}^\tau} \right\rVert_{\F}
    \\
    & = \left\lVert P_{\Phi \cap Q_{1:k}^{\perp}}
    T_{\tau} P_{Q_{1:k}} P_{\Gamma_{1:k}^\tau} \right\rVert_{\F} \\
    & \leq \left\lVert P_{\Phi \cap Q_{1:k}^{\perp}}
    T_{\tau} P_{Q_{1:k}} \right\rVert_{\F}.
\end{align}
where we have used the fact that $\Gamma_{1:k}^\tau$ is an invariant subspace of $P_{\Phi} T_{\tau} P_{\Phi}$.
Next,
we introduce the subspace $\Eta_{1:k} = \spa_{1 \leq i \leq k} \eta_i$,
which
is orthogonal to $\Phi \cap Q_{1:k}^\perp$.
Then,
\begin{equation}
    \left\lVert
    P_{\Eta_{1:k}^\perp}
    P_{Q_{1:k}}
    \right\rVert_{\F}
    = \left\lVert
    P_{\Eta_{1:k}}
    P_{Q_{1:k}^\perp}\right\rVert_{\F} \\
    = \left\lVert
    P_{\Eta_{1:k}}
    P_{\Phi^\perp}\right\rVert_{\F} \\
    = d_{\F}\left(\spa_{1 \leq i \leq k} \eta_i, \Phi\right).
\end{equation}
To complete the theorem, it is enough to show
\begin{equation}
\label{eq:enough_to_show}
    \left\lVert 
    P_{\Phi \cap Q_{1:k}^{\perp}}
    T_{\tau} P_{Q_{1:k}} 
    \right\rVert_{\F}
    \leq 
    \frac{e^{-\sigma_{k+1} \tau}}{2} 
    \left\lVert
    P_{\Eta_{1:k}^\perp}
    P_{Q_{1:k}}
    \right\rVert_{\F}.
\end{equation}

To prove equation \eqref{eq:enough_to_show},
we apply a useful property of the Frobenius norm.
For bounded linear operators $A$ and $B$, 
if it is true that $\left\lVert Au \right\rVert \leq \left\lVert Bu \right\rVert$
for all $u$, 
then it follows that $\left\lVert A \right\rVert_{\F}
\leq \left\lVert B \right\rVert_{\F}$ \cite{horn2012matrix}.
Using this property, it is sufficient to prove
\begin{equation} 
    \left\lVert P_{\Phi \cap Q_{1:k}^{\perp}}
    T_{\tau} q\right\rVert
    \leq \frac{e^{-\sigma_{k+1} \tau}}{2} \left\lVert P_{\Eta_{1:k}^\perp} q
    \right\rVert,
    \quad q \in Q_{1:k}.
\end{equation}
Moreover, it is sufficient to prove that
\begin{equation}
\label{eq:inequality}
    \left|\left<\phi, T_t q\right>\right|
    \leq \frac{e^{-\sigma_{k+1} \tau}}{2},
\end{equation}
for all $\phi \in \Phi \cap Q_{1:k}^\perp$
and $q \in Q_{1:k}$ with $\left\lVert \phi \right\rVert
= \left\lVert P_{\Eta_{1:k}^{\perp}} q \right\rVert = 1$.
We observe
\begin{equation}
    \left\lVert P_{\Eta_{1:k}^{\perp}} \left(\phi \pm q\right) \right\rVert^2
    = \left\lVert \phi \pm P_{\Eta_{1:k}^{\perp}} q \right\rVert^2
    = 2 \pm 2\left< \phi, P_{\Eta_{1:k}^{\perp}} q\right>
    = 2 \pm 2\left< \phi, q\right>
    = 2.
\end{equation}
Using the polarization identity
and the fact that $\Eta_{1:k}^\perp$ is an invariant subspace of $T_{\tau}$, conclude
\begin{align}
    \label{eq:calc_start}
    \left|\left<\phi,
    T_{\tau}
    q\right>\right|
    &= \left|\left<P_{\Eta_{1:k}^{\perp}} \phi,
    T_{\tau} P_{\Eta_{1:k}^{\perp}}
    q\right>\right| \\
    &= \left|\frac{1}{4} \left<P_{\Eta_{1:k}^{\perp}}\left(\phi + q\right),
    T_{\tau} 
    P_{\Eta_{1:k}^{\perp}} \left(\phi + q\right)\right>
    -
    \frac{1}{4} \left<P_{\Eta_{1:k}^{\perp}}
    \left(\phi - q\right),
    T_{\tau} 
    P_{\Eta_{1:k}^{\perp}}
    \left(\phi - q \right)\right>\right|  \\
    \label{eq:calc_end}
    & \leq \frac{1}{2}
    \left\lVert P_{\Eta_{1:k}^{\perp}}
    T_{\tau}
    P_{\Eta_{1:k}^{\perp}}
    \right\rVert_2 \\
    & \leq \frac{e^{-\sigma_{k+1} \tau}}{2}.
\end{align}
\end{proof}

\begin{proof}[Proof of \cref{thm:remarkable2}]
When $\lambda_k^{\tau} > e^{-\sigma_{k+1}}$,
\cref{prop:mainbound} allows us to calculate
\begin{align}
    d_{\F}^2\left(\spa_{1 \leq i \leq k} \eta_i, \Phi\right)
    & \leq d_{\F}^2\left(\spa_{1 \leq i \leq k} \gamma_i^\tau, \spa_{1 \leq i \leq k} \eta_i \right) \\
    &= \left\lVert P_{\Phi^\perp} P_{\Eta_{1:k}} \right\rVert_{\F}^2
    + \left\lVert P_{\left(\Gamma_{1:k}^\tau\right)^\perp}
    P_{\Phi} P_{\Eta_{1:k}} \right\rVert_{\F}^2 \\
    &= \left\lVert P_{\Phi^\perp} P_{\Eta_{1:k}} \right\rVert_{\F}^2
    + \left\lVert P_{\left(\Gamma_{1:k}^\tau\right)^\perp}
    P_{Q_{1:k}} P_{\Eta_{1:k}} \right\rVert_{\F}^2 \\
    &\leq 
    \left(1 + \frac{1}{4} \left|\frac{e^{-\sigma_{k+1} \tau}}{\lambda_k^\tau - e^{-\sigma_{k+1} \tau}}\right|^2\right)
    d_{\F}^2\left(\spa_{1 \leq i \leq k} \eta_i, \Phi\right).
\end{align}
\end{proof}

\begin{proof}[Proof of \eqref{eq:remarkable_subspaces}]
When there is a gap between $\left\{\sigma_j, \ldots, \sigma_k\right\}$ and all other $\sigma_i$ values,
\cref{prop:mainbound} shows that
$\spa_{1 \leq i \leq j} \gamma_i^\tau \rightarrow \spa_{1 \leq i \leq k} q_i$ and $\spa_{1 \leq i \leq k} \gamma_i^\tau \rightarrow \spa_{1 \leq i \leq k} q_i$ in the limit $\tau \rightarrow \infty$.
By applying \cref{lem:quadratic}, we verify $\spa_{j \leq i \leq k} \gamma_i^{\tau} \rightarrow \spa_{j \leq i \leq k} q_i$.
\end{proof}


\subsection{Exponential speed of convergence}{\label{sub:error_multiplier}}

In this section, we verify the last statement in
\cref{thm:remarkable}
that the top $k$ idealized VAC eigenfunctions satisfy
\begin{equation}
    d_{\F}\left(\spa_{1 \leq i \leq k} \gamma_i^{\tau},
    \spa_{1 \leq i \leq k} q_i\right) \frac{\lambda_k^\tau}{\lambda_{k+1}^{\tau}}
    \rightarrow
    \left|\frac{\left<\eta_{k+1}, q_k\right>}
    {\left<\eta_{k+1}, q_{k+1}\right>}\right|
    \tag{\ref{eq:equivalent}}
\end{equation}
in the limit $\tau \rightarrow \infty$, provided there is a gap between $\sigma_k$ and 
all other $\sigma_i$ values and a gap between $\sigma_{k+1}$ and all other $\sigma_i$ values.

\begin{proof}[Proof of \cref{eq:equivalent}]
If there are fewer orthogonalized projection functions $\left(q_i\right)_{1 \leq i \leq p}$
compared to basis functions $\left(\phi_i\right)_{1 \leq i \leq n}$,
select
additional functions $\left(q_i\right)_{p + 1 \leq i \leq n}$
so that $\left(q_i\right)_{1 \leq i \leq n}$
is a complete orthonormal basis for $\Phi$.
Then,
\begin{align}
    d_{\F}\left(\spa_{1 \leq i \leq k} \gamma_i^{\tau},
    \spa_{1 \leq i \leq k} q_i\right) &= \left\lVert 
    \left(\sum\nolimits_{i = k+1}^n q_i \left<q_i, \cdot \right>\right)
    \left(\sum\nolimits_{j = 1}^k \gamma_j^\tau
    \left<\gamma_j^\tau, \cdot \right>\right)
    \right\rVert_{\F} \\
    &= \left\lVert 
    \sum\nolimits_{i = k+1}^n \sum\nolimits_{j = 1}^k
    q_i \left<q_i, \gamma_j^\tau\right> \left<\gamma_i^{\tau}, \cdot\right>
    \right\rVert_{\F} \\
    &= \left(\sum\nolimits_{i = k+1}^n \sum\nolimits_{j = 1}^k
    \left<q_i, \gamma_j^\tau\right>^2\right)^{1 \slash 2}.
\end{align}
The terms $\left<q_i, \gamma_j^\tau\right>$
are determined by the eigenvalue equation
\begin{equation}
\label{eq:eigenvalue_eqn}
    \lambda_j^\tau \left<q_i, \gamma_j^\tau\right>
    = \left<q_i, T_{\tau} \gamma_j^\tau\right>
    = \sum\nolimits_{l=1}^n \left<q_i, T_{\tau} q_l\right>
    \left<q_l, \gamma_j^\tau\right>.
\end{equation}
As $\tau \rightarrow \infty$,
\cref{thm:knyazev} and the calculations in \eqref{eq:calc_start}-\eqref{eq:calc_end} imply
\begin{equation}
    1 \slash \lambda_j^\tau
    = \mathcal{O}\left(
    e^{\sigma_j \tau}\right)
    \quad\text{and}\quad \left|\left<q_i, T_{\tau} q_l\right>\right| \leq \min\left\{e^{-\sigma_i \tau}, e^{-\sigma_l \tau}\right\}.
\end{equation}
Setting $\epsilon = \min\left\{\sigma_{k+2} - \sigma_k, \sigma_{k+1} - \sigma_{k-1}\right\}$ and sending $\tau \rightarrow \infty$, we find
\begin{align}
    d_{\F}\left(\spa_{1 \leq i \leq k} \gamma_i^{\tau},
    \spa_{1 \leq i \leq k} q_i\right)
    &= \left|\left<q_{k+1}, \gamma_k\right>\right|
    + \mathcal{O}\left(e^{-\epsilon \tau}\right) \\
    &= \left|\frac{1}{\lambda_k^\tau}
    \sum\nolimits_{l = 1}^n
    \left<q_{k+1}, T_{\tau} q_l\right>
    \left<q_l, \gamma_k^\tau\right>\right|
    + \mathcal{O}\left(e^{-\epsilon \tau}\right).
\end{align}
Lastly, using the facts that $\gamma_k^\tau \rightarrow q_k$ and $\lambda_{k+1}^\tau e^{\sigma_{k+1} \tau} \rightarrow \left<\eta_{k+1}, q_{k+1}\right>^2$,
\begin{align}
    d_{\F}\left(\spa_{1 \leq i \leq k} \gamma_i^{\tau},
    \spa_{1 \leq i \leq k} q_i\right) \frac{\lambda_k^\tau}{\lambda_{k+1}^{\tau}}
    &= \frac{\left<q_{k+1}, T_{\tau} q_k\right>}{\lambda_{k+1}^{\tau}} \left(1 + o\left(1\right)\right) \\
    &= \frac{e^{-\sigma_{k+1} \tau} \left<\eta_{k+1}, q_{k+1}\right> \left<\eta_{k+1}, q_k\right>}
    {e^{-\sigma_{k+1} \tau} \left<\eta_{k+1}, q_{k+1}\right>^2}
    \left(1 + o(1)\right) \\
    &= \frac{\left<\eta_{k+1}, q_k\right>}
    {\left<\eta_{k+1}, q_{k+1}\right>}
    \left(1 + o(1)\right).
\end{align}
\end{proof}


\subsection{Formulas for the estimation error}{\label{sub:matrix_perturbation}}

In this section, we verify the formulas for the estimation error 
that are given in \cref{thm:samplingError}.

\begin{proof}[Proof of \cref{thm:samplingError}]
First, define matrices
\begin{align}
    & \Lambda\left(\tau\right) = \diag \left\{
    \begin{pmatrix} \lambda_1^\tau & \cdots & \lambda_n^\tau
    \end{pmatrix} \right\},
    & V\left(\tau\right) = 
    \begin{pmatrix} 
    v_1\left(\tau\right) & \dots & v_n\left(\tau\right)
    \end{pmatrix}.
\end{align}
Due to the normalization $\delta_{ij} = \left<\gamma_i^\tau, \gamma_j^\tau\right>$,
we must have
\begin{align}
\label{eq:identities}
    & V\left(\tau\right)^T C\left(0\right) V\left(\tau\right) = I,
    & V\left(\tau\right)^T C\left(\tau\right) V\left(\tau\right) = \Lambda\left(\tau\right).
\end{align}
Therefore, idealized VAC eigenfunctions
$\gamma_i^{\tau}$
and eigenvalues $\lambda_i^{\tau}$
are the eigenfunctions and eigenvalues of the multiplication operator
\begin{equation}
    \sum\nolimits_{i,j=1}^n
    \gamma_i^{\tau}
    \Lambda_{ij}\left(\tau\right) \left<\gamma_j^\tau, \cdot \right>.
\end{equation}
In contrast, VAC eigenfunctions $\hat{\gamma}_i^{\tau}$
and eigenvalues $\hat{\lambda}_i^{\tau}$
are eigenfunctions and eigenvalues of
\begin{equation}
    \label{eq:structural_rep}
    \sum\nolimits_{i,j=1}^n
    \gamma_i^{\tau} \left(V\left(\tau\right)^{-1} \hat{C}\left(0\right)^{-1} \hat{C}\left(\tau\right) V\left(\tau\right)\right)_{ij}
    \left<\gamma_j^\tau, \cdot\right>.
\end{equation}
As $\hat{C}\left(0\right) \rightarrow C\left(0\right)$
and $\hat{C}\left(\tau\right) \rightarrow C\left(\tau\right)$,
we can calculate
\begin{align}
    & C\left(0\right) \hat{C}\left(0\right)^{-1} \hat{C}\left(\tau\right) \\
    &= \left[I + \left[\hat{C}\left(0\right) C\left(0\right)^{-1} - I\right]\right]^{-1} 
    \hat{C}\left(\tau\right) \\
    &= \hat{C}\left(\tau\right) - \left[\hat{C}\left(0\right) C\left(0\right)^{-1} - I\right] C\left(\tau\right)
    + \mathcal{O}\left(\left\lVert \hat{C}\left(0\right) - C\left(0\right)\right\rVert_{\F}^2
    + \left\lVert \hat{C}\left(\tau\right) - C\left(\tau\right)\right\rVert_{\F}^2\right) \\
    &= V\left(\tau\right)^{-T} \left[\Lambda\left(\tau\right) + \hat{L}\left(\tau\right)\right] V\left(\tau\right)^{-1} +
    \mathcal{O}\left( \left\lVert \hat{C}\left(0\right) - C\left(0\right)\right\rVert_{\F}^2
    + \left\lVert \hat{C}\left(\tau\right) - C\left(\tau\right)\right\rVert_{\F}^2\right),
\end{align}
where we have made repeated use of the identities \eqref{eq:identities}.
Multiplying on the left by $V\left(\tau\right)^T$ and on the right by $V\left(\tau\right)$, we find
that VAC eigenspaces are unitarily equivalent to the eigenspaces of the matrix operator
\begin{equation}
    V\left(\tau\right)^{-1} \hat{C}\left(0\right)^{-1} \hat{C}\left(\tau\right) V\left(\tau\right)
    = \Lambda\left(\tau\right) + \hat{L}\left(\tau\right) +
    \mathcal{O}\left( \left\lVert \hat{C}\left(0\right) - C\left(0\right)\right\rVert_{\F}^2
    + \left\lVert \hat{C}\left(\tau\right) - C\left(\tau\right)\right\rVert_{\F}^2\right),
\end{equation}
and the two operators share the same eigenvalues.
\cref{thm:samplingError}
then
follows by applying first-order perturbation bounds \cite{karow2014perturbation} for
eigenvalues and invariant subspaces
of a diagonal matrix $\Lambda\left(\tau\right)$
that is perturbed by a matrix $\hat{L}\left(\tau\right)$.
\end{proof}


\subsection{Distributional formulas for the estimation error}{\label{sub:variance}}

In this section, we verify the distributional formulas for the estimation error that are given in \cref{thm:moments}.

\begin{proof}[Proof of \cref{thm:moments}]

We fix the lag time $t \geq 0$ and the indices $1 \leq i,j \leq n$,
but allow the total trajectory length $T$ to vary.
Then, we write
\begin{align}
& \hat{C}_{ij}\left(t\right) = \frac{\Delta}{T - t} \sum_{s = 0}^{\frac{T - t}{\Delta} - 1} \phi_{ij}\left(X_{s \Delta}\right),
&
\phi_{ij}\left(x, y\right) = \frac{\phi_i\left(x\right) \phi_j\left(y\right) + \phi_i\left(y\right) \phi_j\left(x\right)}{2}.
\end{align}
As $T \rightarrow \infty$, we will proceed to show that 
$\sqrt{T} \left(\hat{C}_{ij}\left(t\right) - C_{ij}\left(t\right)\right)$ converges to an asymptotic normal distribution.

By assumption, $X_t$ is started from the stationary distribution $X_0 \sim \mu$,
so the random variables
$\left(\phi_{ij} \left(X_{s \Delta}, X_{s \Delta + \tau}\right)\right)_{s = 0, 1, \ldots}$ are strictly stationary \cite[pg. 230-231]{meyn2012markov}
with mean $C_{ij}\left(t\right)$.
Moreover, the conditional expectations $\E\left[\left. \phi_{ij}\left(X_{s\Delta}, X_{s\Delta + \tau}\right)\right| X_0 = x\right]$
satisfy
\begin{equation}
\label{eq:negligible}
    \left\lVert \E\left[\left. \phi_{ij}\left(X_{s\Delta}, X_{s\Delta + \tau}\right)\right| X_0 = x\right] - C_{ij}\left(t\right) \right\rVert \leq C e^{-\sigma_2 s \Delta},
    \quad s \geq 0,
\end{equation}
for a constant $C < \infty$ that is independent of $s$.
Condition \eqref{eq:negligible} is an asymptotic negligibility condition that guarantees the validity of the central limit theorem for
$\phi_{ij} \left(X_{s \Delta}, X_{s \Delta + \tau}\right)$.
Using the central limit theorem in \cite[ch. 5]{hall2014martingale},
we prove that
\begin{equation}
    \sqrt{T} \left(\hat{C}_{ij}\left(t\right) - C_{ij}\left(t\right)\right)
    \stackrel{\mathcal{D}}{\rightarrow} \N\left(0, \Delta \sum_{s = -\infty}^{\infty}
    \Cov_{\mu}\left[\phi_{ij}^\tau\left(X_0, X_\tau\right),
    \phi_{ij}^\tau\left(X_{s \Delta}, X_{s \Delta + \tau}\right)\right]\right).
\end{equation}

For simplicity, we have considered the asymptotic distribution of $\sqrt{T} \left(\hat{C}_{ij}\left(t\right) - C_{ij}\left(t\right)\right)$.
However, by the same approach we can prove the asymptotic normality of any linear combination of random variables $\sqrt{T} \left(\hat{C}_{ij}\left(t\right) - C_{ij}\left(t\right)\right)$ involving different values of $i$, $j$, and $t$.
By the Cram{\'e}r-Wold theorem \cite{kallenberg2006foundations}, therefore,
the array 
\begin{equation}
\left[\sqrt{T}\left(\hat{C}\left(0\right) - C\left(0\right)\right), \sqrt{T}\left(\hat{C}\left(\tau\right) - C\left(\tau\right)\right)\right]
\end{equation}
converges to a mean-zero multivariate normal distribution.

To complete the proof of \cref{thm:moments}, we then apply the ``delta method" \cite[pg. 26]{van2000asymptotic}
using the formulas \eqref{eq:eigenvalue_error} and \eqref{eq:asymptotic}.
Since $\sqrt{T} \left\lVert \hat{C}\left(0\right) - C\left(0\right)\right\rVert_{\F}^2$
and $\sqrt{T} \left\lVert \hat{C}\left(\tau\right) - C\left(\tau\right)\right\rVert_{\F}^2$
terms are $\mathcal{O}_p\left(\frac{1}{\sqrt{T}}\right)$ as $T \rightarrow \infty$, these terms are asymptotically negligible.
The matrix
\begin{equation}
    \hat{L}\left(\tau\right) = V\left(\tau\right)^T  \left(\hat{C}\left(\tau\right) - \hat{C}\left(0\right)\right) V\left(\tau\right) -
    V\left(\tau\right)^T \left(\hat{C}\left(0\right)
    - C\left(0\right)\right) V\left(\tau\right) \Lambda\left(\tau\right)
\end{equation}
is a linear combination
of matrices $\hat{C}\left(\tau\right) - C\left(\tau\right)$
and $\hat{C}\left(0\right) - C\left(0\right)$,
with each matrix entry $\hat{L}_{ij}\left(\tau\right)$ satisfying
\begin{equation}
    \sqrt{T} \hat{L}_{ij}\left(\tau\right) \stackrel{\mathcal{D}}{\rightarrow} \N\left(0, \Delta \sum_{s = -\infty}^{\infty}
    \Cov_{\mu}\left[F_{ij}^\tau\left(X_0, X_\tau\right),
    F_{ij}^\tau\left(X_{s \Delta}, X_{s \Delta + \tau}\right)\right]\right).
\end{equation}

\end{proof}



\section{Conclusions}
\label{sec:conclusions}

In this paper,
we have identified and bounded the major error sources of ``the variational approach to conformational dynamics" (VAC)  \cite{noe2013variational, chodera2014markov, noe2017collective, husic2018markov}.
VAC is frequently applied in biomolecular simulation studies to estimate the largest eigenvalues
$e^{-\sigma_1 \tau} \geq e^{-\sigma_2 \tau} \geq \cdots \geq e^{-\sigma_k \tau}$
for the Markov transition operator $T_{\tau}$,
along with the corresponding eigenfunctions
$\eta_1, \eta_2, \ldots, \eta_k$.

We have proved that VAC accurately identifies
subspaces of eigenfunctions $\spa_{j \leq i \leq k} \eta_i$
when three conditions are satisfied:
\begin{enumerate}[leftmargin = *]
    \item{\label{item:condition1}} The values $\left\{\sigma_j, \ldots, \sigma_k\right\}$ are separated from all other $\sigma_i$ values by a spectral gap.
    \item{\label{item:condition2}} The library of basis functions $\left(\phi_i\right)_{1 \leq i \leq n}$ becomes very rich so that linear combinations of basis functions can fully represent $\eta_1, \ldots, \eta_k$.
    \item{\label{item:condition3}} The data set becomes very large
    so that expectations $C_{ij}\left(0\right) = \E_{\mu}\left[\phi_i\left(X_0\right) \phi_j\left(X_0\right)\right]$
    and $C_{ij}\left(\tau\right) = \E_{\mu}\left[\phi_i\left(X_0\right) \phi_j\left(X_\tau\right)\right]$
    are evaluated with vanishing error.
\end{enumerate}

VAC converges for any value of the lag time parameter $\tau > 0$,
yet the choice of lag time can dramatically alter the speed of convergence.
Hence, our main contribution
is to prove error bounds
that explicitly show how error depends on the lag time.
These bounds provide a full theoretical justification for why
limitations in the basis set
contribute to the error at short lag times
and limitations in the data set
contribute to the to error at long lag times.

Our numerical analysis approach is flexible,
and it could be extended
to algorithms besides VAC that 
estimate dynamical quantities of interest
using trajectory data.
A broadly useful approach involves decomposing the total error into
approximation error and estimation error.
Another useful approach 
involves identifying asymptotic formulas for the estimation error.
In future research, it is our goal to rigorously analyze
the approximation and estimation error
for other powerful algorithms used in biochemical simulation 
(e.g., \cite{thiede2019galerkin}).

Lastly, while the main purpose of our work is to deepen theoretical understanding, 
we also provide diagnostic
tools to assess VAC's estimation error
and tune VAC's parameters to reduce this error source.
We present the \emph{VAC condition number} as a tool for identifying 
subspaces of VAC eigenfunctions that are prone to experiencing high estimation error.
We present the \emph{mean squared estimation error}
as a tool for calculating the estimation error
at different lag times.
Motivated by the present study, 
we have also developed an approach for reducing the lag time sensitivity and increasing VAC's robustness
by integrating over a window of lag times \cite{lorpaiboon2020integrated}.
These tools have
direct relevance to the
researchers using VAC,
pointing the way toward a more streamlined lag time selection process
and a more critical
assessment of VAC's error for the future.



\section*{Acknowledgments}
The authors would like to acknowledge helpful conversations with 
Gary Froyland, Chatipat Lorpaiboon and John Strahan.
We would also like to thank the anonymous reviewers for their helpful critiques and suggestions.


\bibliographystyle{siamplain}
\bibliography{references}


\section{Supplement}


\subsection{Figures for the Ornstein-Uhlenbeck process}{\label{sec:ou_figures}}

Here we provide additional information about how \cref{fig:stabilizes} - \cref{fig:implied_timescales} were generated.
These figures show VAC applied to the Ornstein-Uhlenbeck process
\begin{equation}
\mathop{dX} = -X \mathop{dt} + \sqrt{2} \mathop{dW}.
\end{equation}
started from the stationary $\N\left(0, 1\right)$ distribution.
The eigenfunctions of the transition operator $T_t$ are the Hermite polynomials
\begin{equation}
1, x, \frac{x^2 - 1}{\sqrt{2}}, 
\frac{x^3 - 3x}{\sqrt{6}}, \ldots
\end{equation}
with eigenvalues $1, e^{-t}, e^{-2t}, e^{-3t}, \ldots$.

The conditional distribution for the OU process is determined by
\begin{equation}
\Law\left(\left.X_t\right| X_0 = x\right) = \N\left(x e^{-t}, 1 - e^{-2t}\right).
\end{equation}
Therefore, we can simulate the OU process 
in discrete time using the exact evolution equations
\begin{equation}
X_{t + \Delta} = e^{-\Delta} X_t + \xi_t,
\quad \xi_t \sim \N\left(0, 1 - e^{-2 \Delta t}\right).
\end{equation}

When we apply VAC to the OU process,
we use a basis of $n$ indicator functions on disjoint intervals,
namely,
\begin{equation}
\left\{
\mathds{1}_{\left(-\infty, q_1\right)},
\mathds{1}_{\left[q_1, q_2\right)},
\mathds{1}_{\left[q_2, q_3\right)},
\ldots,
\mathds{1}_{\left[q_{n-1}, \infty \right)}
\right\}.
\end{equation}
The boundary points
\begin{equation}
q_0 = -\infty < q_1 < q_2 < \cdots < q_{n-1} < q_n = \infty
\end{equation}
are selected as follows:
\begin{enumerate}
	\item{\label{item:step1}}
	First, we set $q_i = \Phi^{-1}\left(i \slash n\right)$ where 
	\begin{equation}
	\Phi\left(x\right) = \int_{-\infty}^x \frac{e^{-y^2 \slash 2}}{\sqrt{2\pi}} \mathop{dy}
	\end{equation} is the cumulative distribution function for a standard normal random variable,
	and $\Phi^{-1}$ is the inverse cumulative distribution function, also called the quantile function.
	\item{\label{item:step2}}
	Next, we set $q_i \leftarrow q_i + \epsilon$, where $\epsilon$ is an offset parameter that is always $\epsilon = 0.1$ in our figures.
	The offset parameter helps make our examples
	realistic,
	since it would typically be impossible in VAC applications
	to identify quantiles of the equilibrium distribution exactly.
\end{enumerate}

Although many quantities involving the Ornstein-Uhlenbeck process
can be calculated analytically, we used numerical quadrature
to evaluate the integrals
\begin{equation}
\left<\eta_i, \phi_j\right> = \int_{q_{j-1}}^{q_j} 
\eta_i\left(x\right) \frac{e^{-x^2 \slash 2}}{\sqrt{2 \pi}} \mathop{dx}
\end{equation}
and 
\begin{equation}
\left<\phi_i, T_{\tau} \phi_j\right> = \int_{q_{j-1}}^{q_j} 
\left[\Phi\left(\frac{q_i - x e^{-\tau}}{\sqrt{1 - e^{-2 \tau}}}\right) - \Phi\left(\frac{q_{i-1} - x e^{-\tau}}{\sqrt{1 - e^{-2 \tau}}}\right)\right]
\frac{e^{-x^2 \slash 2}}{\sqrt{2 \pi}} \mathop{dx}.
\end{equation}



\subsection{Figures for the double well process}

Here we provide additional information about how \cref{fig:double_well} and \cref{fig:double_well_vac_evals} were generated.
These figures show VAC applied to the process
\begin{equation}
\mathop{dX} = - \frac{1}{2}\sigma \sigma^T \nabla U(X) \mathop{dt} + \sigma \mathop{dW},
\end{equation}
where the potential $U$ and the diffusion matrix 
$\sigma$ are given by:
\begin{align}
& U(x_1, x_2) = 4 x_1^4 - 8 x_1^2 +  x_1 + 0.5 x_2^2,
& \sigma = 
\begin{pmatrix} 2 & 0 \\ -1 & \sqrt{3} \end{pmatrix}.
\end{align}
$X_t$ is a double well process
that spends long time periods in potential wells near $\left(-1,0\right)$ and $\left(1,0\right)$
with rare transitions between wells.
We simulate $X_t$ using the BAOAB-limit integrator presented in Leimkuhler and Matthews \cite{leimkuhler2013rational} with the timestep $\Delta = 10^{-4}$.
We discard the first $t = 10$ time units of each trajectory
to reduce equilibration bias.

To calculate reference values for the true eigenfunctions $\eta_i$ and the idealized VAC matrices $C\left(\tau\right)$,
we use the numerical PDE approach from appendix D of reference \cite{thiede2019galerkin}.
We first construct a grid from $-2$ to $2$ in $x$ and from $-5$ to $5$ in $y$
with grid spacing of $\epsilon = \left( 6 \times 10^{-4} \right)^{-1 / 2}$.
We next construct the transition matrix $P$ for a hopping process
on a grid:
\begin{align}\label{eq:grid_hopping_probabilities}
P(x \pm \epsilon,y) &= \frac{1}
{6\left(1 + \exp\left[U(x\pm\epsilon,y) - U(x, y)\right]\right)}, \\
P(x, y \pm \epsilon) &= 
\frac{1}{6\left(1 + \exp\left[U(x,y\pm\epsilon) - U(x, y)\right]\right)}, \\
P(x\pm \epsilon,y\pm\epsilon) &= \frac{1}{6\left(1 + \exp\left[U(x\pm\epsilon,y\pm\epsilon) - U(x, y)\right]\right)}, \\
P(x \pm \epsilon, y \mp \epsilon) &= 0, \\
P(x,y) &= 1 - P(x + \epsilon, y) - P(x - \epsilon, y)
- P(x, y + \epsilon) - P(x, y - \epsilon).
\end{align}
In the $\epsilon \to 0$ limit, the action of $\frac{24}{\epsilon^2} \left( P - I \right)$ on smooth functions
approximates the action of the infinitesimal generator $L$
for the process $X_t$.

We calculate the eigenfunctions $\eta_i$
using eigenfunctions of $P$.
We calculate the idealized VAC matrix $C\left(\tau\right)$ using the approximation
\begin{align}
C_{ij}\left(\tau\right) 
&= \left<\phi_i, e^{L\tau} \phi_j\right> \\
&\approx \left<\phi_i, \left(I + \frac{\epsilon^2}{24} L\right)^{24 \tau \slash \epsilon^2} \phi_j\right> \\
&\approx \vec{\phi}_i^T D_{\mu} P^{24 \tau \slash \epsilon^2} \vec{\phi}_j,
\end{align}
where $\vec{\phi}_i$ is the vector of $\phi_i$ values evaluated at each gridpoint and $D_\mu$ is a diagonal matrix with the stationary measure at each gridpoint
along the diagonal.


\subsection{Mean squared estimation error}

In \cref{alg:errorcalc},
we describe the procedure
for calculating the mean squared estimation error using data.
The procedure is similar to the one used
to calculate
error bars in Markov Chain Monte Carlo.
Therefore, we take advantage of existing software
for Markov Chain Monte Carlo sampling \cite{foreman2013emcee}
in our implementation.

\begin{algorithm}[h!]
	\caption{Asymptotic estimation error}
	\label{alg:errorcalc}
	\begin{enumerate}[leftmargin = *]
		\item For $1 \leq i,j \leq n$, perform the following calculations.
		\begin{enumerate}
			\item Form the time series $\hat{F}_{ij}^{\tau}\left(X_{s \Delta}, X_{s \Delta + \tau}\right)$
			for $s = 0, 1, \ldots, \frac{T - \tau}{\Delta} - 1$,
			where the function $\hat{F}_{i j}^\tau\left(x, y\right)$
			is given by
			\begin{equation}
			\frac{\hat{\gamma}_{i}^\tau \left(x\right)
				\hat{\gamma}_j^\tau\left(y\right) +
				\hat{\gamma}_{i}^\tau\left(y\right)
				\hat{\gamma}_j^\tau\left(x\right)}{2}
			- \hat{\lambda}_j^\tau \frac{
				\hat{\gamma}_{i}^\tau\left(x\right)
				\hat{\gamma}_j^\tau\left(x\right) +
				\hat{\gamma}_{i}^\tau\left(y\right)
				\hat{\gamma}_j^\tau\left(y\right)}{2}.
			\end{equation}
			\item{\label{item:stepb}} For $s = 0, 1, \ldots, \frac{T - \tau}{\Delta} - 1$, calculate the autocovariance terms $\hat{R}_{ij}\left(s \Delta\right)$ given by
			\begin{equation}
			\frac{1}{\frac{T - \tau}{\Delta} - s} \sum_{r=0}^{\frac{T - \tau}{\Delta} - s - 1}
			\hat{F}_{i j}^{\tau}\left(X_{r \Delta}, X_{r \Delta + \tau}\right)
			\hat{F}_{i j}^{\tau}\left(X_{\left(r + s\right) \Delta}, X_{\left(r + s\right) \Delta + \tau}\right).
			\end{equation}
			\item{\label{item:stepc}} Use the approach in \cite[pg.143-145]{sokal1997monte} to
			determine a truncation threshold $K$
			such that $\hat{R}_{ij}\left(s \Delta\right) \approx 0$ for $s > K$, and set
			\begin{equation}
			\hat{V}_{ij}\left(\tau\right)^2
			=
			\frac{\Delta}{T}
			\left(1 + 2 \sum_{s = 1}^K
			\hat{R}_{ij}\left(s \Delta\right)\right).
			\end{equation}
		\end{enumerate}
		\item Estimate the mean squared estimation error using
		\begin{align}
		& \E\left|\hat{\lambda}_i^{\tau} - \lambda_i^{\tau}\right|^2
		\approx \hat{V}_{ii}\left(\tau\right)^2, \\
		& \E\left|d_{\F}\left(\spa_{j \leq i \leq k} \hat{\gamma}_i^{\tau}, \spa_{j \leq i \leq k} \gamma_i^{\tau}\right)\right|^2
		\approx 
		\sum_{\substack{l < j \\ \text{or } l > k}}
		\sum_{m = j}^k
		\frac{\hat{V}_{l m}\left(\tau\right)^2}
		{\left|\hat{\lambda}_l^{\tau} - \hat{\lambda}_m^{\tau}\right|^2}.
		\end{align}
	\end{enumerate}
\end{algorithm}


\subsection{Rayleigh-Ritz approximation bounds}

In this section, we re-derive the classic approximation bounds for the
Rayleigh-Ritz method first presented in \cite{knyazev1985sharp} and \cite[pg. 992]{knyazev1997new}.
Our first step is to verify the inequality
\begin{equation}
1 - d_2^2\left(\spa_{1 \leq i \leq k} \eta_i, \Phi\right)
\leq \frac{\lambda_k^{\tau}}{e^{-\sigma_k \tau}} \leq 1
\end{equation}
that appears in the statement of \cref{thm:knyazev}.
\begin{proof}[Proof of equation \eqref{eq:eigenvalue_bound}]
	As in the proof of \cref{thm:remarkable},
	the upper bound 
	\begin{equation}
	\lambda_k^{\tau}
	= \max_{\dim\left(\Eta\right) = k, \Eta \subseteq \Phi}
	\min_{\eta \in \Eta} \frac{\left<\eta, T_{\tau} \eta\right>}{\left<\eta, \eta\right>}
	\leq e^{-\sigma_k \tau}
	\end{equation}
	follows directly from the min-max principle.
	
	The lower bound on $\lambda_k^{\tau}$
	follows trivially if $d_2\left(\spa_{1 \leq i \leq k} \eta_i, \Phi\right) = 1$.
	If $d_2\left(\spa_{1 \leq i \leq k}, \Phi\right) < 1$,
	then we define subspaces
	$\Eta_{1:k} = \spa_{1 \leq i \leq k} \eta_i$ and $Q_{1:k} = P_{\Phi} \spa_{1 \leq i \leq k} \eta_i$.
	For any $q \in Q_{1:k}$ with $\left\lVert q \right\rVert = 1$, we calculate
	\begin{align}
	e^{\sigma_k \tau} &\leq \frac{\left<T_{\tau} P_{\Eta_{1:k}} q, P_{\Eta_{1:k}} q\right>}
	{\left<P_{\Eta_{1:k}} q, P_{\Eta_{1:k}} q\right>} \\
	&= \frac{\left<T_{\tau} P_{\Eta_{1:k}} q, P_{\Eta_{1:k}} q\right>}
	{1 - \left\lVert P_{\Eta_{1:k}^{\perp}} q \right\rVert^2} \\
	& \leq \frac{\left<P_{\Eta_{1:k}}
		T_{\tau} P_{\Eta_{1:k}} q, q\right>
		+ \left<P_{\Eta_{1:k}^{\perp}} T_{\tau} P_{\Eta_{1:k}^{\perp}} q, q\right>}
	{1 - \left\lVert P_{\Eta_{1:k}^{\perp}} P_{Q_{1:k}} \right\rVert^2} \\
	&= \frac{\left<T_{\tau} q, q\right>}
	{1 - d_2^2\left(\Eta_{1:k}, \Phi\right)}.
	\end{align}
	We conclude that
	\begin{equation}
	\left(1 - d_2^2\left(\Eta_{1:k}, \Phi\right)\right) e^{\sigma_k \tau}
	\leq \frac{\left<q, T_{\tau} q\right>}
	{\left<q, q\right>},
	\quad q \in Q_{1:k}.
	\end{equation}
	The lower bound then follows by applying the min-max principle.
\end{proof}

It remains to verify the inequality
\begin{equation}
1 \leq \frac{d_{\F}^2\left(\spa\limits_{1 \leq i \leq k} \gamma_i^{\tau},
	\spa\limits_{1 \leq i \leq k} \eta_i \right)}
{d_{\F}^2\left(\spa\limits_{1 \leq i \leq k} \eta_i, \Phi\right)}
\leq 1 + \frac{\left\lVert P_{\Phi^{\perp}}
	T_{\tau} P_{\Phi} \right\rVert_2^2}
{\left|e^{-\sigma_k \tau} - \lambda_{k+1}^{\tau}\right|^2}
\tag{\ref{eq:subspace_bound}}
\end{equation}
that appears in the statement of \cref{thm:knyazev}.
\begin{proof}[Proof of equation \eqref{eq:subspace_bound}]
	We define subspaces $\Eta_{j:k} = \spa_{j \leq i \leq k} \eta_i$ and
	$\Gamma_{j:k}^\tau = \spa_{j \leq i \leq k} \gamma_i^\tau$ for all $1 \leq j < k \leq n$.
	Then, it follows
	\begin{equation}
	1 \leq \frac{d_{\F}^2\left(\spa\limits_{1 \leq i \leq k} \gamma_i^{\tau},
		\spa\limits_{1 \leq i \leq k} \eta_i \right)}
	{d_{\F}^2\left(\spa\limits_{1 \leq i \leq k} \eta_i, \Phi\right)}
	= \frac{\left\lVert P_{\Eta_{1:k}} P_{\left(\Gamma_{1:k}^\tau\right)^\perp} \right\rVert_{\F}^2}
	{\left\lVert P_{\Eta_{1:k}} P_{\Phi^{\perp}} \right\rVert_{\F}^2} 
	\leq 1 + 
	\frac{\left\lVert P_{\Eta_{1:k}} P_{\Gamma_{k+1:n}^\tau} \right\rVert_{\F}^2}
	{\left\lVert P_{\Eta_{1:k}} P_{\Phi^{\perp}} \right\rVert_{\F}^2}.
	\end{equation}
	
	It remains to bound the distance between $\Eta_{1:k}$
	and the idealized VAC subspace $\Gamma_{k+1:n}^\tau$.
	To bound this distance, we apply the Davis-Kahan lemma
	as in the proof of \cref{thm:remarkable}.
	The spectrum of $P_{\Eta_{1:k}} \left.T_{\tau}\right|_{\Eta_{1:k}}$
	lies in
	$\left[e^{-\sigma_k \tau}, \infty\right)$,
	while the spectrum of $P_{\Gamma_{k+1:n}^{\tau}} \left.T_{\tau} \right|_{\Gamma_{k+1:n}^\tau}$ lies in
	$\left(-\infty, \lambda_{k+1}^\tau\right]$.
	Therefore, the spectral gap is at least $e^{-\sigma_k \tau} - \lambda_{k+1}^\tau$.
	It follows that
	\begin{align}
	\left(e^{-\sigma_k \tau} - \lambda_{k+1}^{\tau}\right)
	\left\lVert 
	P_{\Eta_{1:k}}
	P_{\Gamma_{k+1:n}^\tau}
	\right\rVert_{\F} 
	&\leq \left\lVert 
	P_{\Eta_{1:k}} P_{\Gamma_{k+1:n}^\tau} T_{\tau} P_{\Gamma_{k+1:n}^\tau}
	- P_{\Eta_{1:k}} 
	T_{\tau} P_{\Eta_{1:k}}
	P_{\Gamma_{k+1:n}^\tau}
	\right\rVert_{\F} \\
	&= \left\lVert 
	P_{\Eta_{1:k}} P_{\Phi} T_{\tau} 
	P_{\Gamma_{k+1:n}^\tau}
	- P_{\Eta_{1:k}} T_{\tau}
	P_{\Gamma_{k+1:n}^\tau}
	\right\rVert_{\F} \\
	&= \left\lVert
	P_{\Eta_{1:k}} P_{\Phi^{\perp}} T_{\tau} P_{\Gamma_{k+1:n}^\tau}
	\right\rVert_{\F} \\
	&\leq \left\lVert
	P_{\Eta_{1:k}} P_{\Phi^{\perp}}
	\right\rVert_{\F}
	\left\lVert P_{\Phi^{\perp}}
	T_{\tau} P_{\Phi}
	\right\rVert_{2}.
	\end{align}
	where we have used the fact that $\Gamma_{k+1:n}^\tau$ is an invariant subspace of $P_{\Phi} T_{\tau} P_{\Phi}$
	and $\Eta_{1:k}$ is an invariant subspace of $\tau$.
\end{proof}


\subsection{Sharper bounds on the lag-time-independent error}

Here we prove an elegant bound on the lag-time-independent error.

\begin{proposition}{\label{prop:orthogonalized2}}
	The lag-time-independent error satisfies
	\begin{equation}
	1 \leq \frac{d_{\F}^2\left(\spa\limits_{j \leq i \leq k} q_i, \spa\limits_{j \leq i \leq k}\eta_i \right)}
	{d_{\F}^2\left(\spa\limits_{j \leq i \leq k} \eta_i, \Phi\right)}
	\leq \frac{1}
	{1 - d_2^2\left(\spa\limits_{1 \leq i \leq j-1} \eta_i, \Phi\right)}.
	\end{equation}
\end{proposition}
\begin{proof}
	To verify the upper bound,
	it is enough to prove
	\begin{equation}
	\label{eq:our_goal}
	\left\lVert P_{Q_{j:k}^{\perp}} \eta \right\rVert^2
	\leq \frac{\left\lVert P_{\Phi^{\perp}} \eta \right\rVert^2}
	{1 - \left\lVert P_{\Phi^{\perp}}
		P_{\Eta_{1:j-1}}\right\rVert_2^2},
	\quad \eta \in \Eta_{j:k}.
	\end{equation}
	Moreover, observing that
	\begin{equation}
	\left\lVert P_{Q_{j:k}^{\perp}} \eta \right\rVert^2
	= \left\lVert P_{Q_{1:j-1}} \eta \right\rVert^2
	+ \left\lVert P_{\Phi^{\perp}} \eta \right\rVert^2,
	\end{equation}
	it is enough to prove
	\begin{equation}
	\label{eq:enough_to_prove}
	\left\lVert P_{Q_{1:j-1}} \eta \right\rVert^2 = \left\lVert P_{Q_{j:k}^{\perp}} \eta \right\rVert^2
	\leq \frac{\left\lVert P_{\Phi^{\perp}} \eta \right\rVert^2
		\left\lVert P_{\Phi^{\perp}}
		P_{\Eta_{1:j-1}}\right\rVert_2^2}
	{1 - \left\lVert P_{\Phi^{\perp}}
		P_{\Eta_{1:j-1}}\right\rVert_2^2},
	\quad \eta \in \Eta_{j:k}.
	\end{equation}
	
	If $P_{Q_{1:j-1}} \eta = 0$,
	then equation \eqref{eq:enough_to_prove} follows trivially.
	Therefore, we consider $\eta \in \Eta_{j:k}$ such that
	$P_{Q_{1:j-1}} \eta \neq 0$.
	Then, there is a function $\eta^{\prime} \in \Eta_{1:j-1}$ with
	\begin{equation}
	\label{eq:big_combo3}
	P_{\Phi} \eta^{\prime}
	= \frac{P_{Q_{1:j-1}} \eta}
	{\left\lVert P_{Q_{1:j-1}} \eta\right\rVert}.
	\end{equation}
	We can bound the norm of $\eta^{\prime}$ by observing
	\begin{equation}
	\left\lVert \eta^{\prime}\right\rVert^2
	= \left\lVert P_{\Phi} \eta^{\prime}\right\rVert^2
	+ \left\lVert P_{\Phi^{\perp}}
	\eta^{\prime}\right\rVert^2
	\leq 1
	+ \left\lVert P_{\Phi^{\perp}} P_{\Eta_{1:j-1}}\right\rVert_2^2 
	\left\lVert \eta^{\prime}\right\rVert^2.
	\end{equation}
	This gives the norm bound
	\begin{equation}
	\label{eq:big_combo4}
	\left\lVert \eta^{\prime}\right\rVert^2
	\leq \frac{1}
	{1 - \left\lVert P_{\Phi^{\perp}} P_{\Eta_{1:j-1}}\right\rVert_2^2 
		\left\lVert \eta^{\prime}\right\rVert^2}.
	\end{equation}
	Using the norm bound and the orthogonality of $\eta \in \Eta_{j:k}$ and $\eta^\prime \in \Eta_{1:j-1}$,
	we conclude
	\begin{align}
	\left\lVert P_{Q_{1:j-1}} \eta \right\rVert^2
	&= \left\lVert \left<P_{Q_{1:j-1}} \eta, P_{\Phi} \eta^{\prime}\right> P_{\Phi} \eta^{\prime} \right\rVert^2 \\
	&= \left<P_{Q_{1:j-1}} \eta, P_{\Phi} \eta^{\prime}\right>^2 \\
	&= \left< P_{\Phi} \eta,
	P_{\Phi} \eta^{\prime}\right>^2
	\\
	&= \left< P_{\Phi^{\perp}} \eta,
	P_{\Phi^{\perp}} \eta^{\prime}\right>^2 \\
	&\leq \left\lVert P_{\Phi^{\perp}} \eta \right\rVert^2
	\left\lVert P_{\Phi^{\perp}} \eta^{\prime} \right\rVert^2 \\
	&\leq \frac{\left\lVert P_{\Phi^{\perp}} \eta \right\rVert^2
		\left\lVert P_{\Phi^{\perp}} P_{\Eta_{1:j-1}} \right\rVert_2^2}
	{1 - \left\lVert P_{\Phi^{\perp}} P_{\Eta_{1:j-1}}\right\rVert_2^2}.
	\end{align}
\end{proof}

\end{document}

%% file: shared.tex
\usepackage{lipsum}
\usepackage{amsfonts}
\usepackage{graphicx}
\usepackage{epstopdf}
\usepackage{algorithmic}
\usepackage{subcaption} 
\usepackage{dsfont} 
\ifpdf
  \DeclareGraphicsExtensions{.eps,.pdf,.png,.jpg}
\else
  \DeclareGraphicsExtensions{.eps}
\fi

\usepackage{enumitem}
\setlist[enumerate]{leftmargin=.5in}
\setlist[itemize]{leftmargin=.5in}

\newsiamremark{remark}{Remark}
\newsiamremark{hypothesis}{Hypothesis}
\newsiamremark{assumption}{Assumption}
\crefname{hypothesis}{Hypothesis}{Hypotheses}
\newsiamthm{claim}{Claim}

\headers{Error bounds for dynamical spectral estimation}{R. J. Webber, E. H. Thiede, D. Dow, A. R. Dinner, J. Weare}

\title{Error bounds for dynamical spectral estimation\thanks{Submitted to the editors May 5, 2020.\funding{RJW was supported by the National Science Foundation award DMS-1646339.
EHT was supported by the Molecular Software Sciences Institute Software Fellows program.
EHT, ARD, and JW were supported by the National Institutes of Health award R01GM109455.
ARD and JW were supported by the National Institutes of Health award R35GM136381. 
JW was supported by the Advanced Scientific Computing Research Program within the DOE Office of Science through award DE-SC0020427.
Computing resources were provided by the University of Chicago Research Computing Center.}}}

\author{Robert J. Webber\thanks{
Courant Institute of Mathematical Sciences, New York University, New York, NY 
(\email{rw2515@nyu.edu}, \email{weare@nyu.edu}).}
\and Erik H. Thiede\thanks{
Department of Chemistry, University of Chicago, Chicago, IL
(\email{thiede@uchicago.edu}, \email{dinner@uchicago.edu}).}
\and Douglas Dow\thanks{
Department of Mathematics, University of Chicago, Chicago, IL
(\email{ddow11@uchicago.edu}).}
\and Aaron R. Dinner\footnotemark[2]
\and Jonathan Weare\footnotemark[3]}

\usepackage{amsopn}
\DeclareMathOperator{\diag}{diag}
\DeclareMathOperator{\Eta}{H}
\DeclareMathOperator{\N}{\mathcal{N}}
\DeclareMathOperator{\F}{F}

\DeclareMathOperator{\Cov}{Cov}

\DeclareMathOperator*{\spa}{span}

\DeclareMathOperator{\corr}{corr}

\DeclareMathOperator{\Law}{Law}
\DeclareMathOperator{\E}{E}

\makeatletter
\newcommand*{\addFileDependency}[1]{
  \typeout{(#1)}
  \@addtofilelist{#1}
  \IfFileExists{#1}{}{\typeout{No file #1.}}
}
\makeatother

\newcommand*{\myexternaldocument}[1]{%
    \externaldocument{#1}%
    \addFileDependency{#1.tex}%
    \addFileDependency{#1.aux}%
}
